\documentclass[11pt,reqno,a4paper,twoside]{extarticle}

\usepackage[osf]{newtxtext}

\usepackage[left=3cm,right=3cm,top=2.2cm,bottom=2.2cm]{geometry}

\usepackage{graphicx, varioref, amsfonts}
\usepackage{amsthm,amssymb,amsmath,mathrsfs,mathtools,stmaryrd} 
\usepackage{dsfont}
\usepackage{upgreek} 
\usepackage{esint} 
\usepackage{breakurl}
\usepackage[full]{textcomp}
\usepackage{empheq}
\usepackage{enumerate,enumitem}
\usepackage{hyperref}
\usepackage{soul}
\usepackage{cancel}
\usepackage{amscd}
\usepackage{srcltx}
\usepackage{ifthen}
\usepackage{float}
\usepackage{color}
\usepackage{comment}
\usepackage{chngcntr}
\usepackage{etoolbox}
\usepackage{fancyhdr}
\usepackage{caption}

\usepackage{titlefoot}

\usepackage{authblk} 
\theoremstyle{plain}
\newtheorem{lemma}{Lemma}[section]
\newtheorem{proposition}{Proposition}[section]
\newtheorem{corollary}{Corollary}[section]
\newtheorem{theorem}{Theorem}[section]

\theoremstyle{definition}

\newtheorem{remark}{Remark}[section]

\newlist{todolist}{itemize}{2}
\setlist[todolist]{label=$\square$}
\usepackage{pifont}
\mathtoolsset{showonlyrefs}

\begin{document}

\title{Deterministic Control of Stochastic Reaction-Diffusion Equations}
\newcommand\shorttitle{Deterministic Control of SPDEs}
\date{\today}

\author{Wilhelm Stannat}
\author{Lukas Wessels}
\newcommand\authors{Wilhelm Stannat and Lukas Wessels}

\affil{\small Technische Universit\"at Berlin}

\maketitle

\unmarkedfntext{\textit{Mathematics Subject Classification (2020) ---} Primary: 93E20, 49K45; Secondary: 49M05, 60H15, 35K57.}

\unmarkedfntext{\textit{Keywords and phrases ---} Stochastic reaction-diffusion equations, variational approach, optimal control, stochastic Schl\"ogl model, stochastic Nagumo equation, nonlinear conjugate gradient descent.}

\unmarkedfntext{\textit{Email}: \textbullet$\,$ stannat@math.tu-berlin.de $\,$\textbullet$\,$ wessels@math.tu-berlin.de}


\begin{abstract}
  We consider the control of semilinear stochastic partial differential equations (SPDEs) via deterministic controls. In the case of multiplicative noise, existence of optimal controls and necessary conditions for optimality are derived. In the case of additive noise, we obtain a representation for the gradient of the cost functional via adjoint calculus. The restriction to deterministic controls and additive noise avoids the necessity of introducing a backward SPDE. Based on this novel representation, we present a probabilistic nonlinear conjugate gradient descent method to approximate the optimal control, and apply our results to the stochastic Schl\"ogl model. We also present some analysis in the case where the optimal control for the stochastic system differs from the optimal control for the deterministic system.
\end{abstract}

\tableofcontents

\section{Introduction}
In this paper our objective is to investigate the optimal control of the semilinear SPDE
\begin{equation}\label{stateequation}
	\begin{cases}
		\mathrm{d} u^g_t = \left [ \Delta u^g_t + f\left ( u^g_t \right ) + b(t)g(t) \right ] \mathrm{d} t + \sigma(t,u^g_t) \mathrm{d}W^Q_t &\text{on } L^2(\Lambda)\\
		u^g_0(x) = u^0(x) & x\in \Lambda
	\end{cases}
\end{equation}
on bounded domains $\Lambda \subset \mathbb{R}$ with $f$ satisfying a one-sided Lipschitz condition. Here, 
the control $g$ is deterministic. Precise assumptions on the coefficients of \eqref{stateequation} will be 
stated at the beginning of the following section. In the case where $f(u) = ku(1-u)(u-a)$ for $k > 0$ and 
$a\in (0,1)$, equation \eqref{stateequation} is called the stochastic Schl\"ogl model. 

We will be interested in the optimal control of \eqref{stateequation} w.r.t. the following quadratic cost functional
\begin{equation}
	\begin{split}
		J(u^g,g) := & \mathbb{E} \left [ \frac{c_{\overline{\Lambda}}}{2} \int_0^T \int_{\Lambda} \left ( u^g_t(x) - u_{\overline{\Lambda}}(t,x) \right )^2 \mathrm{d}x \mathrm{d}t \right ]\\
		& + \mathbb{E} \left [ \frac{c_T}{2} \int_{\Lambda} \left ( u^g_T(x) - u^T(x) \right )^2 \mathrm{d}x \right ] + \frac{\lambda}{2} \int_0^T \int_{\Lambda} g^2(t,x) \mathrm{d}x \mathrm{d}t,
	\end{split}
\end{equation}
where $u_{\overline{\Lambda}}:[0,T]\times\Lambda \to \mathbb{R}$ and $u^T:\Lambda \to \mathbb{R}$ are a given running and terminal reference profile, respectively, and $c_{\overline{\Lambda}}, c_T, \lambda\geq 0$ are constants. The optimal control of the deterministic counterpart of \eqref{stateequation} (i.e. $\sigma\equiv 0$) has 
been well studied in the existing literature (see the monograph \cite{troeltzsch2010}). In particular, the 
optimal control of the deterministic Schl\"ogl model has been studied in a series of papers by Tr\"oltzsch, 
Ryll et al. (\cite{buchholz2013}, \cite{ryll2016}, \cite{ryll2017}). 

Recent years have seen a rising interest in the optimal control of SPDEs. Whereas there exists already a 
quite substantial literature on the dynamic programming approach to the optimal control of SPDEs (see, e.g., 
the monographs \cite{cerrai20012} and \cite{fabbri2017}, and in particular \cite{cerrai2001} and \cite{masiero2008} for the case of stochastic reaction-diffusion systems) direct variational methods have been much less applied. 

Results concerning existence of optimal controls of nonlinear SPDEs have first been obtained in \cite{lisei2002} in the case of the stochastic Navier-Stokes equation, see also the recent work \cite{cordoni2018} for a discussion of existence of optimal controls for the FitzHugh-Nagumo system with additive noise. Necessary first order conditions for optimality are discussed by Fuhrman et al. in \cite{fuhrman2018, fuhrman2016} within the mild approach to SPDEs. However, in \cite{fuhrman2018} the authors work under the assumption that the G\^ateaux derivative of the nonlinearity is bounded (cf. Assumption 2.1), which does not include the Schl\"ogl model; in \cite{fuhrman2016} the nonlinearity does cover the Schl\"ogl model, but the noise is only additive in this work. Since both these works formulate the problem in the setting of mild solutions to SPDEs, it does not seem tractable to generalize those results to a polynomial nonlinearity while considering a multiplicative noise term. In \cite{cordoni2019} the authors derive first order necessary conditions using a rescaling method which exploits a certain structure of the state equation. The problem of sufficient conditions for optimal controls has been investigated in 
\cite{oksendal2005}. In this paper, the author derives a sufficient maximum principle for a class of 
quasilinear SPDEs with a one-dimensional noise term. 

In the present paper we will be interested in the optimal control of \eqref{stateequation} within the variational approach to SPDEs. With a view towards the efficient numerical approximation we will restrict to deterministic controls, and from Section \ref{gradientsection} on also to additive noise. The restriction to additive noise allows us to avoid the backward SPDE for the adjoint state and to obtain a conceptually much simpler representation in terms of a backward random PDE, instead. However, using this adjoint state for numerical approximations would lead to non-adapted controls. A further restriction to deterministic controls allows us to derive a deterministic gradient in terms of the simpler adjoint state (cf. Theorem \ref{gradient}). This representation of the gradient gives rise to more efficient numerical approximations of (trivially adapted, since deterministic) optimal open-loop controls. We illustrate our approach in the case of the stochastic Schl\"ogl model.

The paper is organized as follows. In Section \ref{wellposed} we state precise assumptions for our analysis, show the well-posedness of the optimal control problem, and prove the existence of an optimal control. In Section \ref{firstorder} we prove the G\^ateaux differentiability of the solution map and the cost functional with respect to the control, and derive a necessary condition for a control to be locally optimal. In Section \ref{gradientsection} we restrict ourselves to the case of additive noise. In this setting, we derive a representation for the gradient of the cost functional as well as an equation for the adjoint state that is later on used in the numerical approximation of locally optimal solutions. Furthermore, we deduce the Stochastic Minimum Principle from the necessary conditions from the previous section. In Section \ref{gradientdescent} we present a probabilistic gradient descent method for the approximation of an optimal control. In Section \ref{application} we apply our results to two examples of the stochastic Schl\"ogl model. In the first example, we show how to accelerate traveling waves and change their direction of travel (cf. Subsection \ref{steering}). The second example is one situation, where the 
optimal control for the stochastic equation apparently differs from the optimal control for the deterministic counterpart (cf. Subsection \ref{spdecase}). Since we are not able to give a rigorous proof in this case, we also consider in Subsection \ref{sdecase} the simplified setting of a stochastic ordinary differential equation, where one can rigorously prove that the optimal control for the stochastic case and its deterministic counterpart are actually different. 

\section{General Setting and Well-Posedness of the Optimal Control Problem}\label{wellposed}
Consider the stochastic partial differential equation \eqref{stateequation} with homogeneous Neumann boundary conditions, where $\Lambda \subset \mathbb{R}$ is a bounded domain, $T>0$ is fixed, $ ( W_t^Q )_{t\in [0,T]}$ is a $Q$-Wiener process for some nonnegative, symmetric trace class operator $Q:L^2(\Lambda)\to L^2(\Lambda)$ on a given filtered probability space $( \Omega, \mathcal{F}, \left ( \mathcal{F}_t )_{t\in [0,T]} , \mathbb{P} \right )$, $u^0 \in L^6(\Omega ,\mathcal{F}_0 , \mathbb{P} ; L^2(\Lambda))$, $b\in L^{\infty}([0,T]\times\Lambda )$, $\sigma: [0,T]\times L^2(\Lambda) \to L(L^2(\Lambda))$ is Fr\'{e}chet differentiable for every fixed $t\in[0,T]$ and satisfies for all $t\in[0,T]$ and $u,v \in L^2(\Lambda)$
\begin{subequations}
	\begin{align}\label{lipschitzsigma}
		\left \|(\sigma(t,u)-\sigma(t,v))\circ \sqrt{Q} \right \|^2_{\text{HS}(L^2(\Lambda))}  
		& \leq C \|u-v\|^2_{L^2(\Lambda)},\\
		\label{growthsigma} \left \|\sigma(t,u)\circ \sqrt{Q} \right \|^2_{\text{HS}(L^2(\Lambda))} &\leq C \left (1+\|u\|^2_{L^2(\Lambda)} \right ),\\
		\label{frechetsigma} \left \|\sigma^{\prime}(t,u)v\circ \sqrt{Q} \right \|^2_{\text{HS}(L^2(\Lambda))} &\leq C \|v\|^2_{L^2(\Lambda)},
	\end{align}
\end{subequations} 
for some constant $C\in\mathbb{R}$, where $\| \cdot \|_{\text{HS}(L^2(\Lambda))}$ denotes the Hilbert-Schmidt norm on the space of all Hilbert-Schmidt operators on $L^2(\Lambda)$. Furthermore, $f:\mathbb{R} \to \mathbb{R}$ is continuously differentiable satisfying $f(0)=0$,
\begin{equation}\label{onesidedlipschitzcondition}
	\sup_{x\in\mathbb{R}} f^{\prime}(x) < \infty,
\end{equation}
and for all $x\in\mathbb{R}$
\begin{equation}\label{growthcondition}
	|f^{\prime}(x)| < C(1+|x|^2),
\end{equation}
for some constant $C\in\mathbb{R}$.
\begin{remark}
	\begin{enumerate}
		\item[1.] Notice that the upper bound of the derivative implies a one-sided Lipschitz condition, i.e. there exists a constant $\widetilde{\text{Lip}}_f\in\mathbb{R}$ such that
		\begin{equation}
			(f(u)-f(v))(u-v)\leq \widetilde{\text{Lip}}_f (u-v)^2,
		\end{equation}
		for all $u,v\in \mathbb{R}$.
		\item[2.] The nonlinearity in the Schl\"ogl equation satisfies these conditions since the leading coefficient of the polynomial is negative and the derivative is a polynomial of degree 2.
		\item[3.] A possible choice for the diffusion coefficient would be the Nemytskii operator associated with $\sigma:\mathbb{R} \to \mathbb{R}$,
		\begin{equation}
			\sigma(u) := \bar \sigma \min \{ u(u-1),M \},
		\end{equation}
		for some constants $\bar \sigma, M>0$. In the stochastic Schl\"ogl model, this choice imposes noise in particular on the wave front of the resulting traveling wave. For a more detailed discussion, see \cite{lord2014}, Example 10.2.
	\end{enumerate}
\end{remark}

Considering the Gelfand triple
\begin{equation}
	H^1(\Lambda) \subset L^6(\Lambda) \subset \left (H^1(\Lambda)\right )^{\ast},
\end{equation}
under appropriate assumptions on the control $g$ (to be specified later), the existence of a variational solution to equation \eqref{stateequation} in the space
\begin{equation}\label{solutionspace}
	E := L^2([0,T]\times\Omega, \mathrm{d}t \otimes \mathbb{P}; H^1(\Lambda)) \cap L^6(\Omega, \mathbb{P} ; C([0,T]; L^2(\Lambda)))
\end{equation}
is assured (see e.g. \cite{liu2015}, Example 5.1.8). 

Our objective is to study the optimal control problem associated with the state equation \eqref{stateequation}. Let $I_1 :L^6(\Omega; C([0,T]; L^2(\Lambda))) \to \mathbb{R}$ be given by
\begin{equation}\label{i1}
	I_1(v) \;:=\; \mathbb{E} \left [ \frac{c_{\overline{\Lambda}}}{2} \int_0^T \int_{\Lambda} \left ( v(t,x) - u_{\overline{\Lambda}}(t,x) \right )^2 \mathrm{d}x \mathrm{d}t + \frac{c_T}{2} \int_{\Lambda} \left ( v(T,x) - u^T(x) \right )^2 \mathrm{d}x \right ]
\end{equation}
and $I_2 : L^2([0,T]\times \Lambda ) \to \mathbb{R}$
\begin{equation}\label{i2}
	I_2(g) \;:=\; \frac{\lambda}{2} \int_0^T \int_{\Lambda} g^2(t,x) \mathrm{d}x \mathrm{d}t, 
\end{equation}
where $c_{\overline{\Lambda}}$, $c_T$, $\lambda \geq 0$, $u_{\overline{\Lambda}}\in L^2\left ( [0,T]\times 
\Lambda \right )$, and $u^T\in L^2(\Lambda)$. We want to minimize the cost functional
\begin{equation}
	J(g) \;:=\; I_1(u^g) + I_2(g),
\end{equation}
subject to the state equation \eqref{stateequation}, where
\begin{equation}
	g\in G_{\text{ad}} \;:=\; \left \{ g\in L^6\left ([0,T]\times \Lambda \right ) |\, \|g\|_{L^6([0,T]\times\Lambda)} \leq \kappa \right \},
\end{equation}
for given $\kappa \geq 0$.

\begin{remark} 
	The proof of the G\^ateaux differentiability of $g\mapsto u^g$ (see Proposition \ref{gateauxsolutionmap} 
	below), requires a moment bound of the solution in $L^6(\Omega \times [0,T]\times\Lambda)$ due to   
	the upper bound \eqref{growthcondition} on the derivative $f^\prime$ of the nonlinearity. Therefore the 
	minimal requirement for an admissible control is $g\in L^6([0,T]\times\Lambda)$. 
	
	In the work by Buchholz et al. (\cite{buchholz2013}) on the deterministic case, the set of admissible controls
	\begin{equation}
		\tilde{G}_{\text{ad}} \;:=\; \left \{ g\in L^{\infty}\left ([0,T]\times \Lambda \right ) |\, g_a \leq g(t,x) \leq g_b \text{ for a.a. } (t,x)\in [0,T]\times\Lambda \right \},
	\end{equation} 
	for some $g_a<g_b$ is considered. We could use the same set in our analysis as well.
\end{remark}

Throughout the whole paper, we are going to work under the aforementioned conditions. First we want to show that the control problem is well-posed. In order to do so, we need the following a priori bound for solutions of the state equation \eqref{stateequation}.

\begin{lemma}\label{aprioribound}
	There is a constant $C = C(b,f,\sigma,T,Q,u^0)$ such that for every solution $u^g\in E$ of the state equation \eqref{stateequation} associated with $g\in G_{\text{ad}}$ on the right hand side we have
	\begin{equation}\label{convergence}
		\mathbb{E} \left [ \sup_{t\in [0,T]} \left \| u^g_t \right \|_{L^2(\Lambda)}^6 + \left (\int_0^T \left \| u^g_t \right \|^2_{H^1(\Lambda)} \mathrm{d}t \right )^3 \right ] \leq \;C\left ( 1+ \int_0^T \left \| g(t) \right \|^6_{L^2(\Lambda)} \mathrm{d}t \right ).
	\end{equation}
\end{lemma}

\begin{proof}
	By It\^{o}'s formula (see \cite{liu2015}, Theorem 4.2.5), we have 
	\begin{equation}
		\begin{split}
			&\left \| u^g_t \right \|_{L^2(\Lambda)}^2\\
			= \;& \| u^0 \|_{L^2(\Lambda)}^2+2 \int_0^t {}_{(H^1(\Lambda))^{\ast}}\left \langle \Delta u^g_s , u^g_s \right \rangle_{H^1(\Lambda)} \mathrm{d}s + 2\int_0^t \left \langle f\left ( u^g_s \right ) ,u^g_s \right \rangle_{L^2(\Lambda)} \mathrm{d}s\\
			&+ 2\int_0^t \left \langle b(s) g(s), u^g_s \right \rangle_{L^2(\Lambda)} \mathrm{d}s + \int_0^t \| \sigma(s,u^g_s)\circ \sqrt{Q} \|^2_{\text{HS}(L^2(\Lambda))}\mathrm{d}s\\
			&+2 \int_0^t \left \langle u^g_s,\sigma(s,u^g_s) dW^Q_s \right \rangle_{L^2(\Lambda)}\\
			\leq \;&\| u^0 \|_{L^2(\Lambda)}^2\!-2\! \int_0^t \! \left \| \nabla u^g_s \right \|^2_{L^2(\Lambda)}\! \mathrm{d}s + \left (2 \,\widetilde{\text{Lip}}_f + C +\|b\|_{L^\infty([0,T]\times\Lambda)} \right )\! \int_0^t \! \left \| u^g_s \right \|^2_{L^2(\Lambda)}\! \mathrm{d}s\\
			\label{ito} &+ \|b\|_{L^\infty([0,T]\times\Lambda)} \int_0^t \left \| g(s) \right \|^2_{L^2(\Lambda)} \mathrm{d}s + T |\Lambda| C+ 2 \left|\int_0^t \left \langle u^g_s , \sigma(s,u^g_s) dW^Q_s \right \rangle_{L^2(\Lambda)} \right|,
		\end{split}
	\end{equation}
	where we used the growth bound \eqref{growthsigma} on $\sigma$ and that by the one-sided Lipschitz continuity of $f$ and $f(0)=0$, we have
	\begin{equation}
		\int_0^t \left \langle f(u^g_s) , u^g_s \right \rangle_{L^2(\Lambda)} \mathrm{d}s \leq \widetilde{\text{Lip}}_f \int_0^t \left \| u^g_s \right \|_{L^2(\Lambda)}^2 \mathrm{d}s.
	\end{equation}
	Taking both sides of equation \eqref{ito} to the power $3$, taking the supremum with respect to $t\in[0,T]$, and taking expectations yields
	\begin{equation}
		\begin{split}
			\mathbb{E} \left [ \sup_{t\in[0,T]} \left \| u^g_t \right \|_{L^2(\Lambda)}^6 \right ] \leq C \Bigg ( &1+\int_0^T \mathbb{E} \left [ \sup_{s\in [0,t]} \left \| u^g_s \right \|^6_{L^2(\Lambda)} \right ] \mathrm{d}t + \int_0^T \left \| g(t) \right \|^6_{L^2(\Lambda)} \mathrm{d}t\\
			&+ \mathbb{E} \left [ \sup_{t\in[0,T]} \left |\int_0^t \left \langle u^g_s , \sigma(s,u^g_s) dW^Q_s \right \rangle_{L^2(\Lambda)} \right |^3 \right ] \Bigg ).\label{bound0}
		\end{split}
	\end{equation}
	By Burkholder-Davis-Gundy inequality (see e.g. \cite{karatzas1991}), we get
	\begin{equation}
		\begin{split}
			\mathbb{E} \left [ \sup_{t\in[0,T]} \left | \int_0^t \left \langle u^g_s, \sigma(s,u^g_s) \mathrm{d} W^Q_s \right \rangle_{L^2(\Lambda)} \right |^3 \right ]\\
			\leq C \mathbb{E} \left [ \left \langle \int_0^{\cdot} \langle u^g_s , \sigma(s,u^g_s) \mathrm{d}W^Q_s \rangle_{L^2(\Lambda)} \right \rangle_T^{\frac32} \right ].\label{burkholder}
		\end{split}
	\end{equation}
	Now, we compute the quadratic variation. To this end, let $(e_k)_{k\geq 1}$ be an orthonormal basis of $L^2(\Lambda)$. Then
	\begin{equation}
		\begin{split}
			\left \langle \int_0^{\cdot} \langle u^g_s , \sigma(s,u^g_s) \mathrm{d}W^Q_s \rangle_{L^2(\Lambda)} \right \rangle_T &= \int_0^T \sum_{k=1}^{\infty} | \langle u^g_s , (\sigma(s,u^g_s)\circ \sqrt{Q}) e_k \rangle_{L^2(\Lambda)} |^2 \mathrm{d}s\\
			&\leq \int_0^T \sum_{k=1}^{\infty} \| u^g_s\|_{L^2(\Lambda)}^2 \| (\sigma(s,u^g_s)\circ \sqrt{Q}) e_k \|_{L^2(\Lambda)}^2 \mathrm{d}s\\
			&= \int_0^T \| \sigma(s,u^g_s)\circ \sqrt{Q}\|^2_{\text{HS}(L^2(\Lambda))} \|u^g_s\|^2_{L^2(\Lambda)} \mathrm{d}s.
		\end{split}
	\end{equation}
	Using the linear growth condition \eqref{growthsigma} we obtain together with equation \eqref{burkholder}
	\begin{equation}
		\begin{split}
			\mathbb{E}\! \left [ \sup_{t\in[0,T]} \left | \int_0^t \left \langle u^g_s, \sigma(s,u^g_s) \mathrm{d} W^Q_s \right \rangle_{L^2(\Lambda)} \right |^3 \right ]\! &\leq \!C\, \mathbb{E}\! \left [ \left (\int_0^T\! \|u^g_s\|_{L^2(\Lambda)}^2 + \| u^g_s \|_{L^2(\Lambda)}^4 \mathrm{d}s \right )^{\frac{3}{2}} \right ]\\
			\!&\leq\! C  \left ( 1 + \int_0^T \mathbb{E} \left [ \sup_{s\in [0,t]}\|u^g_s \|_{L^2(\Lambda)}^6 \right ] \mathrm{d}t \right )
		\end{split}
	\end{equation}
	Together with equation \eqref{bound0} and Gr\"onwall's inequality, this yields
	\begin{equation}\label{bound1}
		\mathbb{E} \left [ \sup_{t\in[0,T]} \left \| u^g_t \right \|_{L^2(\Lambda)}^6 \right ] \leq C \left ( 1 + \int_0^T \left \| g(t) \right \|^6_{L^2(\Lambda)} \mathrm{d}t \right ).
	\end{equation}
	Furthermore, from \eqref{ito}, we get
	\begin{equation}
		\begin{split}
			&\mathbb{E} \left [ \left ( \int_0^T \left \| \nabla u^g_t \right \|^2_{L^2(\Lambda)} \mathrm{d}t \right )^3 \right ]\\
			\leq& C\left ( 1+ \mathbb{E} \left [ \int_0^T \left \| u^g_t \right \|_{L^2(\Lambda)}^6 \mathrm{d}t \right ] + \int_0^T \left \| g(t) \right \|^6_{L^2(\Lambda)} \mathrm{d}t \right ).\label{bound2}
		\end{split}
	\end{equation}
	Putting together equations \eqref{bound1} and \eqref{bound2}, we get for some constant\\ $C = C(b,f,\sigma,T,Q,u^0)$
	\begin{equation}
		\mathbb{E} \left [ \left ( \int_0^T \left \| \nabla u^g_t \right \|^2_{L^2(\Lambda)} \mathrm{d}t \right )^3 \right ] \leq C\left ( 1+\int_0^T \left \| g(t) \right \|^6_{L^2(\Lambda)} \mathrm{d}t \right ). 
	\end{equation}
	Together with \eqref{bound1}, this completes the proof.
\end{proof}

As a consequence, the finiteness of all of the integrals appearing in the cost functional $J$ is assured. Furthermore, we get the following corollary.

\begin{corollary}\label{nirenberg}
	Let $E$ be defined as in \eqref{solutionspace}. Every solution $u^g\in E$ of the state equation \eqref{stateequation} associated with $g\in G_{ad}$ on the right hand side is in $L^6(\Omega\times [0,T]\times \Lambda)$.
\end{corollary}

\begin{proof}
	We apply the Gagliardo-Nirenberg interpolation inequality which can be found in \cite{roubicek2013}. This yields for almost all $(t,\omega)\in [0,T]\times\Omega$
	\begin{equation}
		\left \| u^g_t \right \|_{L^6(\Lambda)}^6 \leq C \left \| u^g_t \right \|_{H^1(\Lambda)}^2 \left \| u^g_t \right \|_{L^2(\Lambda)}^4.
	\end{equation}
	Integrating over $[0,T]\times\Omega$ yields
	\begin{equation}
		\begin{split}
			\mathbb{E} \left [ \int_0^T \left \| u^g_t \right \|_{L^6(\Lambda)}^6 \mathrm{d}t \right ] &\leq \mathbb{E} \left [ \int_0^T \left \| u^g_t \right \|_{H^1(\Lambda)}^2 \left \| u^g_t \right \|_{L^2(\Lambda)}^4 \mathrm{d}t \right ]\\
			&\leq \mathbb{E} \left [ \sup_{t\in[0,T]} \left \| u^g_t \right \|_{L^2(\Lambda)}^4 \int_0^T \left \| u^g_t \right \|_{H^1(\Lambda)}^2 \mathrm{d}t \right ]\\
			&\leq \mathbb{E} \left [ \sup_{t\in [0,T]} \left \| u^g_t \right \|_{L^2(\Lambda)}^6 \right ] \mathbb{E} \left [ \left ( \int_0^T \left \| u^g_t \right \|_{H^1(\Lambda)}^2 \mathrm{d}t \right )^3 \right ]\\
			&<\infty,
		\end{split}
	\end{equation}
	where we used H\"older's inequality and Lemma \ref{aprioribound}.
\end{proof}

Next, we show that the solution map of the state equation \eqref{stateequation} is globally Lipschitz continuous.

\begin{proposition}\label{lipschitzcontinuity}
	Let $E$ be defined as in \eqref{solutionspace}. For the solution map
	\begin{align}
		L^{2}\left ([0,T] \times \Lambda\right ) &\to E\\
		g &\mapsto u^g,
	\end{align}
	there exists a constant $C=C(f,b,\sigma,Q,\Lambda,T)\in\mathbb{R}$ such that
	\begin{equation}\label{lipschitz}
		\left \| u^{g_1}_t - u^{g_2}_t \right \|_E^2 \leq C \int_0^T \left \| g_1-g_2 \right \|_{L^2(\Lambda)}^2\mathrm{d}s
	\end{equation}
	In particular, the solution map is Lipschitz continuous from $L^{2}\left ([0,T] \times \Lambda\right )$ to $E$.
\end{proposition}

\begin{proof}
	By It\^{o}'s formula (see \cite{liu2015}, Theorem 4.2.5), we have almost surely
	\begin{equation}
		\begin{split}
			\left \| u^{g_1}_t - u^{g_2}_t \right \|_{L^2(\Lambda)}^2 =2 &\int_0^t { }_{(H^1(\Lambda))^{\ast}} \left \langle \Delta \left ( u^{g_1} - u^{g_2} \right ),  u^{g_1} - u^{g_2}  \right \rangle_{H^1(\Lambda)} \mathrm{d}s\\
			&+ 2\int_0^t \left \langle f\left ( u^{g_1} \right ) - f \left (u^{g_2}\right ) ,u^{g_1} - u^{g_2} \right \rangle_{L^2(\Lambda)} \mathrm{d}s\\
			&+ 2\int_0^t \left \langle b \left (g_1 - g_2 \right ), u^{g_1} - u^{g_2} \right \rangle_{L^2(\Lambda)} \mathrm{d}s\\
			&+ \int_0^t \|(\sigma(s,u^{g_1}_s) - \sigma(s,u^{g_2}_s))\circ \sqrt{Q} \|_{\text{HS}(L^2(\Lambda))}^2 \mathrm{d}s\\
			&+ 2\int_0^t \langle u^{g_1}_s - u^{g_2}_s , (\sigma(s,u^{g_1}_s) - \sigma(s,u^{g_2}_s)) \mathrm{d}W^Q_s \rangle_{L^2(\Lambda)}.
		\end{split}
	\end{equation}
	Using the Lipschitz condition \eqref{lipschitzsigma} and similar arguments as in the proof of Lemma \ref{aprioribound} yields the claim.
\end{proof}

Now we want to prove the existence of an optimal control:

\begin{theorem}\label{atleastonesolution}
	There is at least one optimal solution $g^{\ast}\in G_{\text{ad}}$ such that
	\begin{equation}
		J(g^{\ast}) = \inf_{g\in G_{\text{ad}}} J(g).
	\end{equation}
\end{theorem}

\begin{proof}
	First, we notice that $J$ is nonnegative and hence bounded from below. Let $\left (g_n\right )_{n\in\mathbb{N}} \subset G_{\text{ad}}$ be a minimizing sequence, i.e.
	\begin{equation}
		\lim_{n\to\infty} J(g_n) = \inf_{g\in G_{\text{ad}}} J(g),
	\end{equation}
	and let $u^{g_n} \in E$ denote the unique solution of the state equation \eqref{stateequation} associated with $g_n$ on the right hand side.
	
	Since $\left ( g_n \right )_{n\in\mathbb{N}} \subset G_{\text{ad}}$, $\left ( g_n \right )_{n\in\mathbb{N}}$ is in particular bounded in $L^2\left ( [0,T]\times \Lambda \right )$. Hence, we can extract a weakly convergent subsequence - again denoted by $g_n$ - such that $g_n \rightharpoonup g^{\ast}$ in $L^2([0,T]\times \Lambda )$. The point is now to show that $g^{\ast} \in G_{\text{ad}}$, and $g^{\ast}$ minimizes $J$ in $G_{\text{ad}}$.
	
	Since $G_{\text{ad}}$ is convex and strongly closed, it follows that $G_{\text{ad}}$ is also weakly closed, hence $g^{\ast} \in G_{\text{ad}}$.
	
	In order to show that $g^{\ast}$ minimizes $J$, we first show that $u^{g_n}$ converges strongly to $u^{g^{\ast}}$. In the deterministic case, the a priori bound in Lemma \ref{aprioribound} holds pathwise and we can apply a compact embedding theorem in order to show strong convergence of the solutions. Since we only have the a priori bound under the expectation, we cannot use the same technique. Instead we apply the so called compactness method introduced in \cite{flandoli1995}. Let us sketch this technique here:
	
	From the bound
	\begin{equation}
		\sup_{n\in\mathbb{N}} \mathbb{E} \left [ \sup_{t\in [0,T]} \left \| u^{g_n}_t \right \|^2_{L^2(\Lambda)} + \int_0^T \left \| u^{g_n}_s \right \|_{H^1(\Lambda)}^2\mathrm{d}s \right ] < \infty
	\end{equation}
	we can conclude tightness of the measures $\mathbb{P}^n := \mathbb{P} \circ (u^{g_n})^{-1}$ on $L^2 ([0,T]\times \Lambda)$. Therefore, $(\mathbb{P}^n)_{n\in\mathbb{N}}$ is relatively compact and we can extract a converging subsequence $\mathbb{P}^n \to \mathbb{P}^{\ast}$. It remains to identify the limit $\mathbb{P}^{\ast}$. By the Skorohod embedding theorem there exists a probability space $(\tilde{\Omega}, \tilde{\mathcal{F}},\tilde{\mathbb{P}})$ and a sequence of random variables $(\tilde{u}^{g_n})_{n\in\mathbb{N}}$ and $\tilde{u}^{g^{\ast}}$ defined on $\tilde{\Omega}$ with the same law as $(u^{g_n})_{n\in\mathbb{N}}$ and $u^{g^{\ast}}$, respectively, such that $\tilde{u}^{g_n}\to \tilde{u}^{g^{\ast}}$ strongly in $L^2([0,T]\times\Lambda)$ $\tilde{\mathbb{P}}$-almost surely. Therefore, using the martingale representation theorem, we can identify $\tilde{u}^{g^{\ast}}$ as a solution to our state equation associated with $g^{\ast}$ on the right hand side (see \cite{daprato2014}, Section 8.4 for details).
	
	Now, we split the cost functional into one part that depends on $u^g$ and into one part that depends on $g$. For the first part, $I_1$, we have
	
	\begin{equation}
		\begin{split}
			&\lim_{n\to\infty} I_1(u^{g_n})\\
			=& \lim_{n\to\infty} \mathbb{E} \left [ \frac{c_{\overline{\Lambda}}}{2} \int_0^T \int_{\Lambda} \left ( u^{g_n}_t(x) - u_{\overline{\Lambda}}(t,x) \right )^2 \mathrm{d}x \mathrm{d}t + \frac{c_T}{2} \int_{\Lambda} \left ( u^{g_n}_T(x) - u^T(x) \right )^2 \mathrm{d}x \right ]\\
			=& \lim_{n\to\infty} \tilde{\mathbb{E}} \left [ \frac{c_{\overline{\Lambda}}}{2} \int_0^T \int_{\Lambda} \left ( \tilde{u}^{g_n}_t(x) - u_{\overline{\Lambda}}(t,x) \right )^2 \mathrm{d}x \mathrm{d}t + \frac{c_T}{2} \int_{\Lambda} \left ( \tilde{u}^{g_n}_T(x) - u^T(x) \right )^2 \mathrm{d}x \right ]\\
			\geq& \tilde{\mathbb{E}} \left [ \liminf_{n\to\infty} \left (\frac{c_{\overline{\Lambda}}}{2} \int_0^T \int_{\Lambda} \left ( \tilde{u}^{g_n}_t(x) - u_{\overline{\Lambda}}(t,x) \right )^2 \mathrm{d}x \mathrm{d}t + \frac{c_T}{2} \int_{\Lambda} \left ( \tilde{u}^{g_n}_T(x) - u^T(x) \right )^2 \mathrm{d}x \right ) \right ]\\
			=& \tilde{\mathbb{E}} \left [ \frac{c_{\overline{\Lambda}}}{2} \int_0^T \int_{\Lambda} \left ( \tilde{u}^{g^{\ast}}_t(x) - u_{\overline{\Lambda}}(t,x) \right )^2 \mathrm{d}x \mathrm{d}t + \frac{c_T}{2} \int_{\Lambda} \left ( \tilde{u}^{g^{\ast}}_T(x) - u^T(x) \right )^2 \mathrm{d}x \right ]\\
			=& \mathbb{E} \left [ \frac{c_{\overline{\Lambda}}}{2} \int_0^T \int_{\Lambda} \left ( u^{g^{\ast}}_t(x) - u_{\overline{\Lambda}}(t,x) \right )^2 \mathrm{d}x \mathrm{d}t + \frac{c_T}{2} \int_{\Lambda} \left ( u^{g^{\ast}}_T(x) - u^T(x) \right )^2 \mathrm{d}x \right ]\\
			=& I_1(u^{g^{\ast}}),
		\end{split}
	\end{equation}
	where we used Fatou's Lemma, and exploited that uniqueness in law holds for the state equation \eqref{stateequation} and we have a solution in the space $E$ (see equation \eqref{solutionspace}).
	
	Furthermore, since $I_2$ is continuous and convex, it is also weakly lower semicontinuous, i.e.
	\begin{equation}
		g_n \rightharpoonup g^{\ast} \qquad \implies \qquad \liminf_{n\to\infty} I_2(g_n) \geq I_2(g^{\ast}).
	\end{equation}
	Therefore, we have
	\begin{equation}
		\inf_{g\in G_{\text{ad}}}\! J(g) =\! \lim_{n\to\infty} J(g_n)\geq \lim_{n\to\infty} I_1(u^{g_n}) + \liminf_{n\to\infty} I_2(g_n) \geq I_1(u^{g^{\ast}}) + I_2(g^{\ast}) = J(g^{\ast}),
	\end{equation}
	which completes the proof.
\end{proof}

\begin{remark}
	This proof does not rely on the explicit form of our cost functional. The crucial point is, that the cost functional is sequentially weakly lower semicontinuous.
\end{remark}

\section{First Order Condition for Critical Points}\label{firstorder}
In this section, we are first going to derive the G\^ateaux derivative of the solution map and the cost functional and then prove a necessary condition for a control to be locally optimal.

\begin{proposition}\label{gateauxsolutionmap}
	Let $f:\mathbb{R} \to \mathbb{R}$ satisfy the assumptions of Section \ref{wellposed} and $g\in L^6([0,T]\times\Lambda)$ be fixed. Then, for every $h\in L^6([0,T]\times\Lambda)$, the G\^{a}teaux derivative of the solution map $g\mapsto u^g$, $L^6([0,T]\times\Lambda) \to E$ in direction $h$ is given by the solution of the linear SPDE
	\begin{equation}\label{gateaux}
		\begin{cases}
			\mathrm{d}y_t^h = [\Delta y_t^h + f^{\prime}(u_t^g)y_t^h + b(t)h(t) ]\mathrm{d}t + \sigma^{\prime}(t,u_t^g) y^h_t \mathrm{d}W^Q_t&\text{on } L^2(\Lambda)\\
			y^h(0,x)=0&x\in\Lambda.
		\end{cases}
	\end{equation}
\end{proposition}

\begin{proof}
	The idea for this proof is taken from \cite{marinelli2020}, Theorem 4.4; see also \cite{bensoussan1983}, where the same technique is applied in a similar setting. Let $y^h$ denote the solution of equation \eqref{gateaux} associated with $h$ on the right hand side. Set
	\begin{equation}
		z_{\delta}(t) := \frac{u_t^{g+\delta h} - u_t^g}{\delta} - y_t^h.
	\end{equation}
	We want to show that $z_{\delta}\to 0$ in $L^2(\Omega\times[0,T];H^1(\Lambda)) \cap L^2(\Omega;C([0,T];L^2(\Lambda)))$ as $\delta \to 0$. First notice
	\begin{equation}
		\begin{split}
			z_{\delta}(t) =& \int_0^t \Delta z_{\delta}(s) + \frac{1}{\delta} \left ( f( u_s^{g+\delta h} ) - f(u_s^g) \right ) - f^{\prime} (u_s^g) y_s^h \mathrm{d}s\\
			&+ \int_0^t \frac{1}{\delta} \left ( \sigma(s,u^{g+\delta h}_s ) - \sigma(s,u^g_s) \right ) - \sigma^{\prime}(s,u^g_s) y^h_s \mathrm{d}W_s^Q.\label{z}
		\end{split}
	\end{equation}
	Note that
	\begin{equation}
		\begin{split}
			&\frac{1}{\delta} \left ( f( u_s^{g+\delta h} ) - f(u_s^g) \right ) - f^{\prime} (u_s^g) y_s^h\\
			=& \underbrace{\frac{1}{\delta} \left ( f(u_s^g+\delta y_s^h) - f(u_s^g) \right ) - f^{\prime} (u_s^g) y_s^h}_{=: R_{\delta}(s)} + \underbrace{\frac{1}{\delta} \left ( f( u_s^{g+\delta h} ) - f(u_s^g+\delta y_s^h) \right )}_{=: S_{\delta}(s)}.
		\end{split}
	\end{equation}
	and similarly
	\begin{equation}
		\begin{split}
			&\frac{1}{\delta} \left ( \sigma(s, u_s^{g+\delta h} ) - \sigma(s,u_s^g) \right ) - \sigma^{\prime} (s,u_s^g) y_s^h\\
			=& \underbrace{\frac{1}{\delta} \left ( \sigma(s,u_s^g+\delta y_s^h) - \sigma(s,u_s^g) \right ) - \sigma^{\prime} (s,u_s^g) y_s^h}_{=: \Xi_{\delta}(s)} + \underbrace{\frac{1}{\delta} \left ( \sigma( s,u_s^{g+\delta h} ) - \sigma(s,u_s^g+\delta y_s^h) \right )}_{=: \Sigma_{\delta}(s)}.
		\end{split}
	\end{equation}
	Together with equation \eqref{z}, It\^{o}'s formula (see \cite{liu2015}, Theorem 4.2.5) yields
	\begin{equation}
		\begin{split}
			\frac12 \left \| z_{\delta}(t) \right \|_{L^2(\Lambda)}^2
			=& \int_0^t\! {}_{(H^1(\Lambda))^{\ast}} \left \langle \Delta z_{\delta}(s), z_{\delta}(s) \right \rangle_{H^1(\Lambda)} \mathrm{d}s + \int_0^t \!\left \langle R_{\delta}(s),z_{\delta}(s) \right \rangle_{L^2(\Lambda)} \mathrm{d}s\\
			&+ \int_0^t\! \left \langle S_{\delta}(s),z_{\delta}(s) \right \rangle_{L^2(\Lambda)} \mathrm{d}s + \int_0^t \!\left \langle z_{\delta}(s), \Xi_{\delta}(s) \mathrm{d}W^Q_s \right \rangle_{L^2(\Lambda)}\\
			&+ \int_0^t\! \left \langle z_{\delta}(s), \Sigma_{\delta}(s) \mathrm{d}W^Q_s \right \rangle_{L^2(\Lambda)}\\
			&+ \frac12 \int_0^t \| \left ( \Xi_{\delta}(s) + \Sigma_{\delta} (s) \right ) \circ \sqrt{Q} \|_{\text{HS}(L^2(\Lambda))}^2 \mathrm{d}s.\label{zdelta}
		\end{split}
	\end{equation}
	First notice that
	\begin{equation}
		\int_0^t {}_{(H^1(\Lambda))^{\ast}} \langle \Delta z_{\delta}(s) , z_{\delta}(s) \rangle_{H^1(\Lambda)} \mathrm{d}s = - \int_0^t \|\nabla z_{\delta}(s) \|^2_{L^2(\Lambda)} \mathrm{d}s.
	\end{equation}
	Furthermore, we have $\langle R_{\delta}(s) , z_{\delta}(s) \rangle_{L^2(\Lambda)} \leq (\|R_{\delta}(s) \|_{L^2(\Lambda)}^2 + \| z_{\delta}(s) \|_{L^2(\Lambda)}^2 )/2$, and, since $f$ is one-sided Lipschitz continuous, we have
	\begin{equation}
		\begin{split}
			\left \langle S_{\delta}(s),z_{\delta}(s) \right \rangle_{L^2(\Lambda)} &= \frac{1}{\delta^2} \left \langle f( u_s^{g+\delta h} ) - f(u_s^g+\delta y_s^h) , u_s^{g+\delta h} - (u_s^g+\delta y_s^h) \right \rangle_{L^2(\Lambda)}\\
			&\leq \widetilde{\text{Lip}}_{f} \| z_{\delta}(s) \|_{L^2(\Lambda)}^2.
		\end{split}
	\end{equation}
	For the last term in equation \eqref{zdelta}, we have
	\begin{equation}
		\begin{split}
			&\frac12 \int_0^T \| \left ( \Xi_{\delta}(s) + \Sigma_{\delta} (s) \right ) \circ \sqrt{Q} \|_{\text{HS}(L^2(\Lambda))}^2 \mathrm{d}s\\
			\leq &\int_0^T \| \Xi_{\delta}(s) \circ \sqrt{Q} \|_{\text{HS}(L^2(\Lambda))}^2 \mathrm{d}s + \int_0^T \| \Sigma_{\delta} (s) \circ \sqrt{Q} \|_{\text{HS}(L^2(\Lambda))}^2 \mathrm{d}s,
		\end{split}
	\end{equation}
	where
	\begin{equation}
		\| \Xi_{\delta}(s) \circ \sqrt{Q} \|^2_{\text{HS}(L^2(\Lambda))} \leq \text{tr} Q \,\| \Xi_{\delta}(s) \|_{L(L^2(\Lambda))}^2,
	\end{equation}
	and, by the Lipschitz condition \eqref{lipschitzsigma} on $\sigma$,
	\begin{equation}
		\begin{split}
			\left \| \Sigma_{\delta}(s) \circ \sqrt{Q} \right \|^2_{\text{HS}(L^2(\Lambda))} &= \left \| \left ( \frac{1}{\delta} \left ( \sigma(s,u^{g+\delta h}_s) - \sigma(s,u^g_s+ \delta y^h_s) \right ) \right ) \circ \sqrt{Q} \right \|^2_{\text{HS}(L^2(\Lambda))}\\
			&\leq C \left \| \frac{1}{\delta} \left ( u^{g+\delta h}_s - u^g_s \right ) - y^h_s \right \|^2_{L^2(\Lambda)} = C \| z_{\delta}(s) \|^2_{L^2(\Lambda)}.
		\end{split}
	\end{equation}
	Therefore, taking the supremum with respect to $t\in[0,T]$ in equation \eqref{zdelta} and taking expectations, it follows
	\begin{equation}
		\begin{split}
			&\mathbb{E} \left [ \sup_{t\in [0,T]} \left \| z_{\delta}(t) \right \|_{L^2(\Lambda)}^2 \right ] + \mathbb{E} \left [ \int_0^T \|\nabla z_{\delta}(s) \|^2_{L^2(\Lambda)} \mathrm{d}s \right ]\\
			\leq C \Bigg \{ &\mathbb{E} \left [ \int_0^T \left \| R_{\delta}(s) \right \|^2_{L^2(\Lambda)} \mathrm{d}s \right ] + \int_0^T \mathbb{E} \left [ \sup_{s\in[0,t]} \| z_{\delta}(s) \|^2_{L^2(\Lambda)} \right ] \mathrm{d}t\\
			&+ \mathbb{E} \left [ \sup_{t\in [0,T]} \int_0^t \!\left \langle z_{\delta}(s), \Xi_{\delta}(s) \mathrm{d}W^Q_s \right \rangle_{L^2(\Lambda)} \right ]\\
			&+ \mathbb{E} \left [ \sup_{t\in [0,T]} \int_0^t\! \left \langle z_{\delta}(s), \Sigma_{\delta}(s) \mathrm{d}W^Q_s \right \rangle_{L^2(\Lambda)} \right ] \Bigg \}.\label{itoestimate}
		\end{split}
	\end{equation}
	Using Burkholder-Davis-Gundy inequality, we have
	\begin{equation}
		\begin{split}
			&\mathbb{E} \left [ \sup_{t\in [0,T]} \int_0^t\! \left \langle z_{\delta}(s), \Sigma_{\delta}(s) \mathrm{d}W^Q_s \right \rangle_{L^2(\Lambda)} \right ]\\
			\leq& C \mathbb{E} \left [ \left \langle \int_0^{\cdot}\! \left \langle z_{\delta}(s), \Sigma_{\delta}(s) \mathrm{d}W^Q_s \right \rangle_{L^2(\Lambda)} \right \rangle^{\frac12}_T \right ].\label{burkholder2}
		\end{split}
	\end{equation}
	Now we compute the quadratic variation. To this end, let $(e_k)_{k\geq 1}$ be an orthonormal basis of $L^2(\Lambda)$. Then we have
	\begin{equation}\label{quadratic}
		\begin{split}
			&\left \langle \int_0^{\cdot}\! \left \langle z_{\delta}(s), \Sigma_{\delta}(s) \mathrm{d}W^Q_s \right \rangle_{L^2(\Lambda)} \right \rangle^{\frac12}_T\\
			= &\left ( \int_0^T \sum_{k=1}^{\infty} \left | \left \langle z_{\delta}(s) , (\Sigma_{\delta}(s) \circ \sqrt{Q}) e_k \right \rangle_{L^2(\Lambda)} \right |^2 \mathrm{d}s \right )^{\frac12} \\
			\leq& \frac{\varepsilon}{2} \sup_{t\in [0,T]} \| z_{\delta}(t) \|_{L^2(\Lambda)}^2 + \frac{1}{2\varepsilon} \int_0^T \left \| \Sigma_{\delta}(s) \circ \sqrt{Q} \right \|_{\text{HS}(L^2(\Lambda))}^2 \mathrm{d}s,
		\end{split}
	\end{equation}
	for arbitrary $\varepsilon > 0$. With the same estimates as above for $\| \Sigma_{\delta}(s) \circ \sqrt{Q} \|_{\text{HS}(L^2(\Lambda))}^2$ and with inequality \eqref{burkholder2} this yields
	\begin{equation}
		\begin{split}
			&\mathbb{E} \left [ \sup_{t\in [0,T]} \int_0^t\! \left \langle z_{\delta}(s), \Sigma_{\delta}(s) \mathrm{d}W^Q_s \right \rangle_{L^2(\Lambda)} \right ]\\
			\leq &C\varepsilon \mathbb{E} \left [ \sup_{t\in [0,T]} \| z_{\delta}(t) \|_{L^2(\Lambda)}^2 \right ] + \frac{C}{\varepsilon} \mathbb{E} \left [ \int_0^T \| z_{\delta}(s) \|_{L^2(\Lambda)}^2 \mathrm{d}s \right ].\label{burkholder1}
		\end{split}
	\end{equation}
	Furthermore, with similar calculations as above, we get
	\begin{equation}
		\begin{split}
			&\mathbb{E} \left [ \sup_{t\in [0,T]} \int_0^t \!\left \langle z_{\delta}(s), \Xi_{\delta}(s) \mathrm{d}W^Q_s \right \rangle_{L^2(\Lambda)} \right ]\\
			\leq&C\varepsilon \, \mathbb{E} \left [ \int_0^T \| z_{\delta}(s) \|_{L^2(\Lambda)}^2 \mathrm{d}s \right ] + \frac{C}{\varepsilon} \, \mathbb{E} \left [ \int_0^T \left \| \Xi_{\delta}(s) \circ \sqrt{Q} \right \|_{\text{HS}(L^2(\Lambda))}^2 \mathrm{d}s \right ],\label{burkholder3}
		\end{split}
	\end{equation}
	for arbitrary $\varepsilon>0$. Choosing $\varepsilon>0$ in \eqref{burkholder1} and \eqref{burkholder3} small enough, we get from \eqref{itoestimate}
	\begin{equation}
		\begin{split}
			&\mathbb{E} \left [ \sup_{t\in [0,T]} \left \| z_{\delta}(t) \right \|_{L^2(\Lambda)}^2 \right ] + \mathbb{E} \left [ \int_0^T \|\nabla z_{\delta}(s) \|^2_{L^2(\Lambda)} \mathrm{d}s \right ]\\
			\leq C \Bigg \{ &\int_0^T \mathbb{E} \left [ \sup_{s\in[0,t]} \| z_{\delta}(s) \|^2_{L^2(\Lambda)} \right ] \mathrm{d}t + \mathbb{E} \left [ \int_0^T \left \| R_{\delta}(s) \right \|^2_{L^2(\Lambda)} \mathrm{d}s \right ]\\
			&+ \mathbb{E} \left [ \int_0^T \left \| \Xi_{\delta}(s) \circ \sqrt{Q} \right \|_{\text{HS}(L^2(\Lambda))}^2 \mathrm{d}s \right ] \Bigg \}.
		\end{split}
	\end{equation}
	By Gr\"onwall's inequality, this yields
	\begin{equation}
		\begin{split}
			&\mathbb{E} \left [ \sup_{s\in[0,T]} \| z_{\delta}(s) \|_{L^2(\Lambda)}^2 \right ] + \mathbb{E} \left [ \int_0^T \|\nabla z_{\delta}(s) \|^2_{L^2(\Lambda)} \mathrm{d}s \right ]\\
			\leq &C \left ( \mathbb{E} \left [ \int_0^T \| R_{\delta} (s) \|_{L^2(\Lambda)}^2 \mathrm{d}s \right ] + \mathbb{E} \left [ \int_0^T \left \| \Xi_{\delta}(s) \circ \sqrt{Q} \right \|_{\text{HS}(L^2(\Lambda))}^2 \mathrm{d}s \right ] \right ).
		\end{split}
	\end{equation}
	Since $R_{\delta} \to 0$ as $\delta\to0$ for almost all $(\omega,t,x)\in\Omega\times[0,T]\times\Lambda$, we get by the dominated convergence theorem
	\begin{equation}\label{dominatedconvergence}
		\lim_{\delta\to 0} \mathbb{E}\left [ \int_0^T \| R_{\delta}(t) \|^2_{L^2(\Lambda)} \mathrm{d}t\right ] = 0.
	\end{equation}
	Here, we used that $R_{\delta}$ is dominated in the following way: By assumption \eqref{growthcondition}, Taylor's formula and elementary estimates, we have
	\begin{equation}
		|R_{\delta}| \leq C \left (1+ \left | u^g \right |^3 + \left | y^h \right |^3 \right ).
	\end{equation}
	The boundedness of the right hand side in $L^2( \Omega\times [0,T] \times \Lambda)$ follows immediately from Corollary \ref{nirenberg} (notice that we get the boundedness of $y^h$ in $L^6(\Omega\times [0,T] \times \Lambda )$ by the same arguments as for $u^g$). Furthermore, we have
	\begin{equation}
		\lim_{\delta \to 0} \mathbb{E} \left [ \int_0^T \left \| \Xi_{\delta}(s) \circ \sqrt{Q} \right \|_{\text{HS}(L^2(\Lambda))}^2 \mathrm{d}s \right ] = 0
	\end{equation}
	since by the Lipschitz condition \eqref{lipschitzsigma} on $\sigma$ and the bound on the Fr\'{e}chet derivative \eqref{frechetsigma} of $\sigma$ we have the following bound:
	\begin{equation}
		\begin{split}
			&\left \|\Xi_{\delta}(s)\circ \sqrt{Q} \right \|^2_{\text{HS}(L^2(\Lambda))}\\
			=&\left \| \left ( \frac{1}{\delta} \left ( \sigma(s,u_s^g+\delta y_s^h) - \sigma(s,u_s^g) \right ) - \sigma^{\prime} (s,u_s^g) y_s^h \right )\circ \sqrt{Q} \right \|^2_{\text{HS}(L^2(\Lambda))}\\
			\leq & 2\left \| \frac{1}{\delta} \left ( \sigma(s,u_s^g+\delta y_s^h) - \sigma(s,u_s^g) \right ) \circ \sqrt{Q} \right \|^2_{\text{HS}(L^2(\Lambda))}\! + 2\left \| \sigma^{\prime} (s,u_s^g) y_s^h \circ \sqrt{Q} \right \|^2_{\text{HS}(L^2(\Lambda))}\\
			\leq & C\left ( 1 + \| y^h \|^2_{L^2(\Lambda)} \right ).
		\end{split}
	\end{equation}
	This completes the proof that $z_{\delta}$ converges to $0$ in $L^2([0,T]\times \Omega; H^1(\Lambda))$ and in\\ $L^2(\Omega;C([0,T];L^2(\Lambda)))$. From the definition of $h\mapsto y^h$, it follows immediately that this is linear. Thus, for the G\^ateaux differentiability it remains to show that $h\mapsto y^h$ is continuous. But this follows with the same arguments as in Proposition \ref{lipschitzcontinuity}.
\end{proof}

As a corollary we get the following representation for the G\^{a}teaux derivative of the cost functional.

\begin{corollary}\label{firstgateaux}
	For every $h\in L^6([0,T]\times\Lambda)$, the cost functional\\
	$J: L^6\left ( [0,T]\times \Lambda \right ) \to \mathbb{R}$ is G\^{a}teaux differentiable in the direction $h$ with G\^{a}teaux derivative
	\begin{equation}
		\begin{split}
			&\frac{\partial J(g)}{\partial h} =\mathbb{E} \Bigg [ c_{\overline{\Lambda}} \int_0^T \int_{\Lambda} y^h_t(x) \left ( u_t^{g}(x) - u_{\overline{\Lambda}}(t,x) \right ) \mathrm{d}x \mathrm{d}t\\
			&+ c_T \int_{\Lambda} y^h_T(x) \left ( u_T^{g}(x)-u^T(x) \right ) \mathrm{d}x + \lambda \int_0^T \int_{\Lambda} g(t,x) h(t,x)\mathrm{d}x \mathrm{d}t \Bigg ],
		\end{split}
	\end{equation}
	where $y^h$ denotes the variational solution of the SPDE \eqref{gateaux}.
\end{corollary}

\begin{proof}
	Recall that the cost functional is given by
	\begin{equation}
		J(g) \;:=\; I_1(u^g) + I_2(g),
	\end{equation}
	where
	\begin{equation}
		I_1(v) \!:=\! \mathbb{E} \left [ \frac{c_{\overline{\Lambda}}}{2} \int_0^T \! \int_{\Lambda} \left ( v(t,x) - u_{\overline{\Lambda}}(t,x) \right )^2 \mathrm{d}x \mathrm{d}t + \frac{c_T}{2} \int_{\Lambda} \left ( v(T,x) - u^T(x) \right )^2 \mathrm{d}x \right ]
	\end{equation}
	and
	\begin{equation}
		I_2(g) \;:=\; \frac{\lambda}{2} \int_0^T \int_{\Lambda} g^2(t,x) \mathrm{d}x \mathrm{d}t.
	\end{equation}
	Hence
	\begin{equation}
		\frac{\partial J(g)}{\partial h} = \frac{\partial I_1\left ( u^{g} \right )}{\partial h} + \frac{\partial I_2(g)}{\partial h}.
	\end{equation} 
	Let $g \in L^6\left ( [0,T]\times \Lambda \right )$ be fixed. For $h\in L^6([0,T]\times\Lambda)$, we get for the G\^{a}teaux derivative of $I_2$
	\begin{equation}\label{itwo}
		\frac{\partial I_2(g)}{\partial h} = \lambda \int_0^T \int_{\Lambda} g(t,x) h(t,x)\mathrm{d}x \mathrm{d}t.
	\end{equation}
	On the other hand we get for the G\^{a}teaux derivative of $I_1$
	\begin{equation}
		\frac{\partial I_1(v)}{\partial w} = \mathbb{E} \left [ c_{\overline{\Lambda}} \int_0^T \int_{\Lambda} w \left ( v - u_{\overline{\Lambda}} \right ) \mathrm{d}x \mathrm{d}t + c_T \int_{\Lambda} w \left ( v-u^T \right ) \mathrm{d}x \right ].
	\end{equation}
	Hence, by the chain rule, we get
	\begin{equation}
		\frac{\partial I_1\left ( u^{g} \right )}{\partial h}=\mathbb{E} \left [ c_{\overline{\Lambda}} \int_0^T \int_{\Lambda} \frac{\partial u^{g}}{\partial h} \left ( u^{g} - u_{\overline{\Lambda}} \right ) \mathrm{d}x \mathrm{d}t + c_T \int_{\Lambda} \frac{\partial u^{g}}{\partial h} \left ( u^{g}-u^T \right ) \mathrm{d}x \right ],
	\end{equation}
	which, together with equation \eqref{itwo} and Proposition \ref{gateauxsolutionmap}, completes the proof.
\end{proof}

Now we can state a necessary condition for $J$ to attain a minimum.

\begin{theorem}\label{eulerlagrange}
	Let $J$ attain a (local) minimum at $g^{\ast} \in G_{\text{ad}}$. Then, for every $h\in G_{\text{ad}}$ we have
	\begin{equation}\label{necessarycondition}
		\frac{\partial J(g^{\ast})}{\partial (h-g^{\ast})} \geq 0.
	\end{equation}
\end{theorem}

\begin{proof}
	Let $h\in G_{\text{ad}}$, and set $\delta_t := g^{\ast} + t(h-g^{\ast}) \in G_{\text{ad}}$. Since $g^{\ast}$ is a local minimizer, there exists a $t_0>0$ such that for all $t\in (0,t_0)$ we have
	\begin{equation}
		J(g^{\ast}) \leq J(\delta_t).
	\end{equation}
	This implies
	\begin{equation}
		\frac{1}{t} \left ( J(g^{\ast} + t(h-g^{\ast})) - J(g^{\ast}) \right ) \geq 0.
	\end{equation}
	Letting $t$ tend to zero yields the claim.
\end{proof}

\section{The Gradient of the Cost Functional}\label{gradientsection}
In this section, we are going to derive a representation for the gradient of the cost functional via adjoint calculus. Recall the state equation
\begin{equation}
	\begin{cases}
		\mathrm{d} u^g_t = \left [ \Delta u^g_t + f\left ( u^g_t \right ) + b(t)g(t) \right ] \mathrm{d} t + \sigma(t,u^g_t) \mathrm{d}W^Q_t &\text{on } L^2(\Lambda)\\
		u^g_0(x) = u^0(x) &x\in \Lambda.
	\end{cases}
\end{equation}

In Section \ref{firstorder}, we proved the following representation
\begin{equation}\label{firstgateauxsolutionmap}
	\begin{split}
		&\frac{\partial J(g)}{\partial h} =\mathbb{E} \Bigg [ c_{\overline{\Lambda}} \int_0^T \int_{\Lambda} y^h_t(x) \left ( u_t^{g}(x) - u_{\overline{\Lambda}}(t,x) \right ) \mathrm{d}x \mathrm{d}t\\
		&+ c_T \int_{\Lambda} y^h_T(x) \left ( u_T^{g}(x)-u^T(x) \right ) \mathrm{d}x + \lambda \int_0^T \int_{\Lambda} g(t,x) h(t,x)\mathrm{d}x \mathrm{d}t \Bigg ],
	\end{split}
\end{equation}
where $y^h$ is the variational solution of
\begin{equation}
	\begin{cases}
		\mathrm{d} y_t^h = [\Delta y_t^h + f^{\prime}(u_t^g)y_t^h + b(t)h(t) ] \mathrm{d}t + \sigma^{\prime}(t,u^g_t) y^h_t \mathrm{d}W^Q_t&\text{on } L^2(\Lambda)\\
		y^h_0(x)=0&x\in\Lambda.
	\end{cases}
\end{equation}
The adjoint equation that is used in the existing literature is
\begin{equation}
	\begin{cases}
		-\mathrm{d}p_t = \left [ \Delta p_t + f^{\prime}(u_t^g)p_t + c_{\overline{\Lambda}} \left ( u_t^g - u_{\overline{\Lambda}}(t,\cdot) \right ) + \partial_u \sigma(t,u^g_t)^{\ast} P_t \right ] \mathrm{d}t\\
		\qquad\qquad\qquad- P_t \mathrm{d}W^Q_t &\text{on } [0,T]\times \Lambda\\
		p_T(x) = c_T \left ( u^g_T(x) - u^T(x) \right )&x\in \Lambda,
	\end{cases}
\end{equation}
for some processes $p\in L^2(\Omega\times [0,T];V)$ and $P\in L^2(\Omega\times[0,T];L_2(U,H))$. The derivation of the Stochastic Minimum Principle with this adjoint equation works in our setting as well. However, the numerical approximation of the solution of this adjoint equation is extremely costly. Therefore, we restrict our analysis to the case of additive noise in the state equation.
\begin{equation}
	\begin{cases}
		\mathrm{d} u^g_t = \left [ \Delta u^g_t + f\left ( u^g_t \right ) + b(t)g(t) \right ] \mathrm{d} t + \sigma \mathrm{d}W^Q_t &\text{on } L^2(\Lambda)\\
		u^g_0(x) = u^0(x) &x\in \Lambda,
	\end{cases}
\end{equation}
for some $\sigma \in \mathbb{R}$. In this case, the linearized equation becomes
\begin{equation}\label{gateaux2}
	\begin{cases}
		\partial_t y_t^h = \Delta y_t^h + f^{\prime}(u_t^g)y_t^h + b(t)h(t)&\text{on } L^2(\Lambda)\\
		y^h_0(x)=0&x\in\Lambda,
	\end{cases}
\end{equation}
which is a random partial differential equation (the coefficient $f^{\prime}(u^g_t)$ is random). Now, we introduce the following random backward PDE for the adjoint state.
\begin{equation}\label{adjoint2}
	\begin{cases}
		-\partial_t p_t = \Delta p_t + f^{\prime}(u^g)p_t + c_{\overline{\Lambda}} \left ( u_t^g - u_{\overline{\Lambda}}(t,\cdot) \right ) &\text{on } [0,T]\times \Lambda\\
		p_T(x) = c_T \left ( u^g_T(x) - u^T(x) \right )&x\in \Lambda.
	\end{cases}
\end{equation}
One crucial point for our algorithm is the fact that the adjoint equation is a random backward PDE. The following property of the adjoint state is the main ingredient in the derivation of the gradient of the cost functional.
\begin{lemma}\label{propertyadjointstate}
	Let $p$ be the solution of the adjoint equation \eqref{adjoint2} and let $y^h$ be the solution of equation \eqref{gateaux2} associated with $u^g$. Then we have almost surely for every $h\in L^6([0,T]\times\Lambda)$
	\begin{equation}
		\int_0^T \int_{\Lambda} b(t) h(t) p_t \mathrm{d}x \mathrm{d}t = \int_0^T \int_{\Lambda} c_{\overline{\Lambda}} (u^g_t -u_{\overline{\Lambda}}(t,\cdot)) y_t^h \mathrm{d}x \mathrm{d}t + \int_{\Lambda} c_T (u^g_T - u^T ) y^h_T \mathrm{d}x.
	\end{equation}
\end{lemma}

\begin{proof}
	By the deterministic integration by parts formula, it holds pathwise
	\begin{equation}
		y^h_T p_T - y_0^h p_0 = \int_0^T y^h_t \mathrm{d} p_t + \int_0^T p_t \mathrm{d} y_t^h.
	\end{equation}
	Plugging in equations \eqref{gateaux2} and \eqref{adjoint2}, respectively, this yields
	\begin{equation}
		\begin{split}
			y^h_T c_T (u^g_T -u^T) =& -\int_0^T y^h_t \left ( \Delta p_t +f^{\prime} (u^g_t) p_t + c_{\overline{\Lambda}} (u^g_t - u_{\overline{\Lambda}}(t,\cdot)) \right ) \mathrm{d}t\\
			&+ \int_0^T p_t \left ( \Delta y^h_t + f^{\prime} (u^g_t ) y^h_t +b(t) h(t) \right ) \mathrm{d}t.
		\end{split}
	\end{equation}
	Integrating over $\Lambda$, and integrating the Laplace operator by parts, we get
	\begin{equation}
		\int_{\Lambda} y^h_T c_T (u^g_T - u^T ) \mathrm{d}x = \int_0^T \int_{\Lambda} b(t)h(t)p_t - c_{\overline{\Lambda}} ( u_t^g - u_{\overline{\Lambda}}(t,\cdot) ) y_t^h \mathrm{d}x \mathrm{d}t,
	\end{equation}
	which is the claimed result.
\end{proof}

As a corollary, we get the following representation for the gradient of the cost functional.

\begin{theorem}\label{gradient}
	The gradient of the cost functional is given by
	\begin{equation}
		\nabla J(g)(t,x) = \mathbb{E} \left [ b(t)p_t(x)+\lambda g(t,x) \right ],
	\end{equation}
	where $p$ is the solution of the adjoint equation
	\begin{equation}
		\begin{cases}
			-\partial_t p_t = \Delta p_t + f^{\prime}(u_t^g)p_t + c_{\overline{\Lambda}} \left ( u_t^g - u_{\overline{\Lambda}}(t,\cdot) \right ) &\text{on } [0,T]\times \Lambda\\
			p_T(x) = c_T \left ( u^g_T(x) - u^T(x) \right )&x\in \Lambda.
		\end{cases}
	\end{equation}
\end{theorem}

\begin{proof}
	By Corollary \ref{firstgateaux}, we have
	\begin{equation}
		\begin{split}
			\frac{\partial J(g)}{\partial h} =&\mathbb{E} \Bigg [ c_{\overline{\Lambda}} \int_0^T \int_{\Lambda} y^h_t(x) \left ( u_t^{g}(x) - u_{\overline{\Lambda}}(t,x) \right ) \mathrm{d}x \mathrm{d}t\\
			&+ c_T \int_{\Lambda} y^h_T(x) \left ( u_T^{g}(x)-u^T(x) \right ) \mathrm{d}x + \lambda \int_0^T \int_{\Lambda} g(t,x) h(t,x)\mathrm{d}x \mathrm{d}t \Bigg ],
		\end{split}
	\end{equation}
	where $y^h$ denotes the variational solution of the random PDE \eqref{gateaux}. Now, by Lemma \ref{propertyadjointstate}, this yields
	\begin{equation}
		\frac{\partial J(g)}{\partial h} =\mathbb{E} \Bigg [ \int_0^T \int_{\Lambda}  b(t) h(t) p_t \mathrm{d}x \mathrm{d}t + \lambda \int_0^T \int_{\Lambda} g(t) h(t) \mathrm{d}x \mathrm{d}t \Bigg ],
	\end{equation}
	which completes the proof.
\end{proof}

Furthermore, by plugging this representation into the necessary condition derived in Theorem \ref{eulerlagrange}, we get the Stochastic Minimum Principle.

\begin{theorem}
	Let $J$ attain a (local) minimum at $g^{\ast} \in G_{\text{ad}}$. Then, for every $h\in G_{\text{ad}}$ we have
	\begin{equation}
		\mathbb{E} \left [ \int_0^T \int_{\Lambda} (b(t) p_t(x) +\lambda g^{\ast}(t,x)) (h(t,x) -g^{\ast}(t,x)) \mathrm{d}x \mathrm{d}t \right ] \geq 0.
	\end{equation}
\end{theorem}

\section{Nonlinear Conjugate Gradient Descent}\label{gradientdescent}
Now that we have identified a representation for the gradient, we can apply a probabilistic nonlinear conjugate gradient descent method in order to approximate the optimal control. We are going to briefly sketch our algorithm here. For a survey of nonlinear conjugate gradient descent methods see \cite{hager2006}.

Let the initial control $g_0\in L^6\left ( [0,T] \times \Lambda \right )$ be given, and fix an initial step size $s_0>0$ as well as a stopping criterion $\eta >0$. Then, the next control can be found as follows.
\begin{enumerate}
	\item \label{step1} Solve the state equation
	\begin{equation}
		\begin{cases}
			\mathrm{d} u^{g_n}_t = \left [ \Delta u^{g_n}_t + f\left ( u^{g_n}_t \right ) + b(t)g_n(t) \right ] \mathrm{d} t + \sigma \mathrm{d}W^Q_t &\text{on } L^2(\Lambda)\\
			u^{g_n}_0(x) = u^0(x) &x\in \Lambda
		\end{cases}
	\end{equation}
	for one realization of the noise.
	\item \label{step2} Solve the adjoint equation
	\begin{equation}
		\begin{cases}
			-\partial_t p_t^n = \Delta p_t^n + f^{\prime}(u_t^{g_n})p_t^n + c_{\overline{\Lambda}} \left ( u_t^{g_n} - u_{\overline{\Lambda}}(t,\cdot) \right ) &\text{on } [0,T]\times \Lambda\\
			p^n_T(x) = c_T \left ( u^{g_n}_T(x) - u^T(x) \right )&x\in \Lambda.
		\end{cases}
	\end{equation}
	with the data given by the sample of the solution of the state equation that was calculated in Step \ref{step1}.
	\item Repeat Step \ref{step1} and Step \ref{step2} to approximate
	\begin{equation}
		\nabla J(g_n)(t,x) = \mathbb{E} \left [ b(t)p^n_t(x)+\lambda g_n(t,x) \right ]
	\end{equation}
	via a Monte Carlo method.
	\item The direction of descent is given by $d_n=-\nabla J(g_n)+\beta_n d_{n-1}$, where $\beta_n= \frac{\|\nabla J(g_n)\|}{\|\nabla J(g_{n-1})\|}$. (In the first step, $\beta_1 =0$.)
	\item \label{step5} Compute the new control via $g_{n+1}=g_n + s_n d_n$.
	\item Accept or deny the new control: Again using a Monte Carlo method, we compare the costs under the new control with the costs under the old control. If the new control decreases the costs, we accept the new control and go back to Step \ref{step1}. Otherwise, we decrease the step size $s_{n}=s_n/2$ and then go back to Step \ref{step5}. (In our simulations, it has proven useful to accept the new control even if the costs are non-decreasing, once the step size gets too small, e.g. $s_n < 10^{-4}$.)
	\item Stop if $\| \nabla J(g_n) \| < \eta$, otherwise reset the step size $s_n=s_0$ and go to step \ref{step1}.
\end{enumerate}

\section{Application to Optimal Control of the Stochastic Schl\"ogl Model}\label{application}

In this section we want to present the application of the algorithm that was introduced in Section \ref{gradientdescent} to the stochastic Schl\"ogl model. We are going to investigate two examples. The first one is to control the speed and the direction of travel of the wave developing in the Schl\"ogl model with additive noise; the second one is an example, where the optimal control of the deterministic system differs from the optimal control of the stochastic system. Corresponding results for the deterministic model can be found in the work by Buchholz et al. (see \cite{buchholz2013}).

\subsection{Steering of a Wave Front}\label{steering}

Let us first recall the Schl\"ogl model. We consider the state equation
\begin{equation}
	\begin{cases}
		\mathrm{d} u^{g}_t = \left [ \Delta u^{g}_t + f\left ( u^{g}_t \right ) + b(t)g(t) \right ] \mathrm{d} t + \sigma \mathrm{d}W^Q_t &\text{on } L^2(\Lambda)\\
		u^{g}_0(x) = u^0(x) &\text{in } \Lambda
	\end{cases}
\end{equation}
with homogeneous Neumann boundary conditions, where $b\equiv 1$, and the nonlinearity is of the form $f(u) = ku(u-1)(a-u)$ for some $k>0$, $a\in(0,1)$, and $\sigma\in\mathbb{R}$, i.e. the state equation takes the form
\begin{equation}
	\begin{cases}
		\mathrm{d} u^{g}_t = \left [ \Delta u^{g}_t + ku^{g}_t(u^{g}_t-1)(a- u^{g}_t) + g(t) \right ] \mathrm{d} t + \sigma \mathrm{d}W^Q_t &\text{on } L^2(\Lambda)\\
		u^{g}_0(x) = u^0(x) &\text{in } \Lambda.
	\end{cases}
\end{equation}
In our example, we choose the time-horizon $[0,15]$, the space $\Lambda=[0,20]$, $k=1$, and $a=39/40$. These choices lead to two stable steady states, $u\equiv 0$ and $u\equiv 1$. As initial condition we choose
\begin{equation}
	u^0(x) = \left ( 1+\exp\left (-\frac{\sqrt{2}}{2}(x-5) \right ) \right )^{-1}
\end{equation}
In the deterministic case we get a traveling wave of the form $u^0(x+ct)$, where $c=\sqrt{2}(\frac12 - a)$ (cf. \cite{chen1992}). Figure \ref{00stochnocontrol} shows the solution in the deterministic case, and Figure \ref{10stochnocontrol} shows one realization of the solution in the stochastic case with $\sigma=0.5$.
\begin{figure}[H]
	\centering
	\begin{minipage}{0.45\textwidth}
		\centering
		\captionsetup{width=1\linewidth}
		\includegraphics[width=1.1\textwidth]{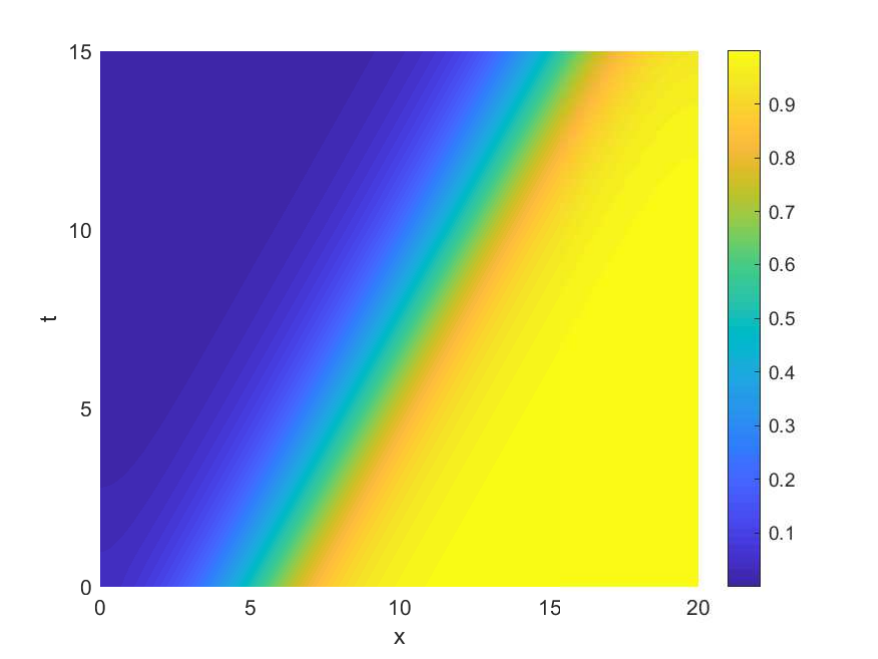}
		\caption[width=2.2\textwidth]{Solution without Control in the Deterministic Case}\label{00stochnocontrol}
	\end{minipage}\hfill
	\begin{minipage}{0.45\textwidth}
		\centering
		\captionsetup{width=1\linewidth}
		\hspace*{-1cm} 
		\includegraphics[width=1.1\textwidth]{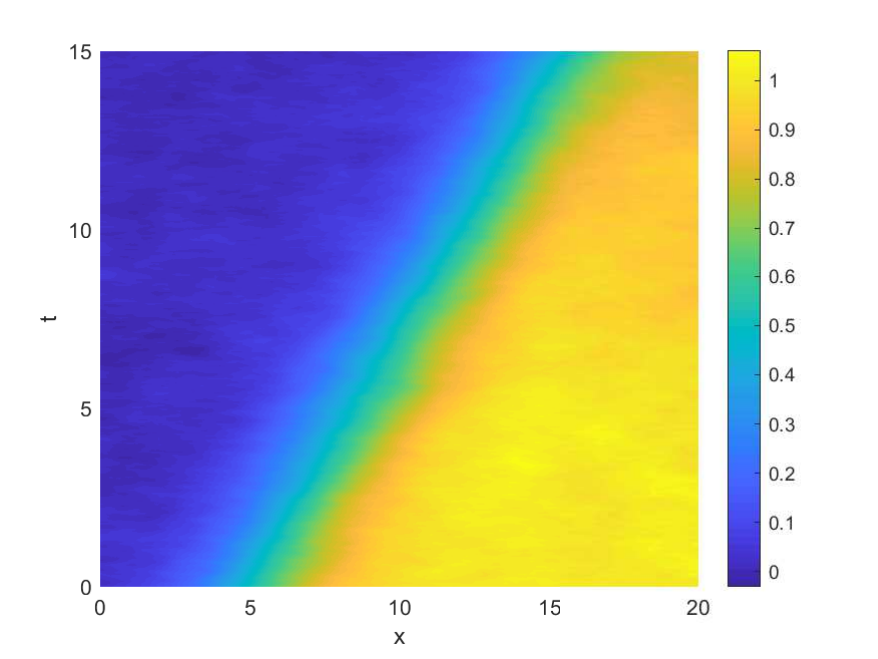}
		\hspace*{-1cm}
		\caption{Solution without Control in the Stochastic Case, $\sigma = 0.5$}\label{10stochnocontrol}
	\end{minipage}
\end{figure}
We can see that the traveling wave slowly travels to the right. Our objective is now to first speed up the wave and then change the direction of travel. To this end, we consider the cost functional given by
\begin{equation}
	J(g) = \mathbb{E} \left [ \frac{c_{\overline{\Lambda}}}{2} \int_0^T \int_{\Lambda} \left ( u^g_t(x)-u_{\overline{\Lambda}}(t,x) \right )^2 \mathrm{d}x \mathrm{d}t\right ]
\end{equation}
where $c_{\overline{\Lambda}}=1$, and the reference profile $u_{\overline{\Lambda}}$ is given by
\begin{equation}
	u_{\overline{\Lambda}}(t,x)= \begin{cases}
		\left (1+\exp\left (-\frac{\sqrt{2}}{2}(x-t-5) \right ) \right )^{-1},&t\leq \frac{T}{2}\\
		\left ( 1+\exp\left (-\frac{\sqrt{2}}{2}(x-(T-t)-5) \right ) \right )^{-1},&t>\frac{T}{2}
	\end{cases},
\end{equation}
for $(t,x)\in[0,T]\times\Lambda$.

With the algorithm from Section \ref{gradientdescent} we can approximate the optimal control. Let us apply the algorithm to the stochastic case with $\sigma=0.5$, the stopping criterion $\eta=0.05$ and $100$ Monte Carlo simulations for the approximation of the gradient. One realization of the solution with applied optimal control is displayed in Figure \ref{10stochoptimalcontrol}. Figure \ref{10stochcontrolfunction} shows the corresponding optimal control.

\begin{figure}[H]
	\centering
	\begin{minipage}{0.45\textwidth}
		\centering
		\captionsetup{width=1\linewidth}
		\includegraphics[width=1.1\textwidth]{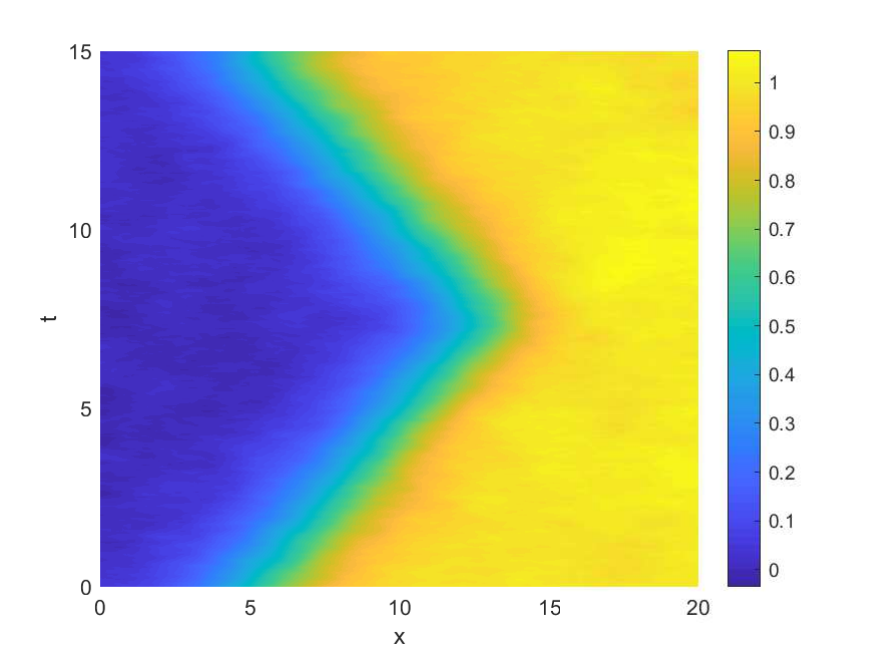}
		\caption{Solution with Optimal Control, $\sigma=0.5$}\label{10stochoptimalcontrol}
	\end{minipage}\hfill
	\begin{minipage}{0.45\textwidth}
		\centering
		\captionsetup{width=1\linewidth}
		\hspace*{-0.5cm}
		\includegraphics[width=1.1\textwidth]{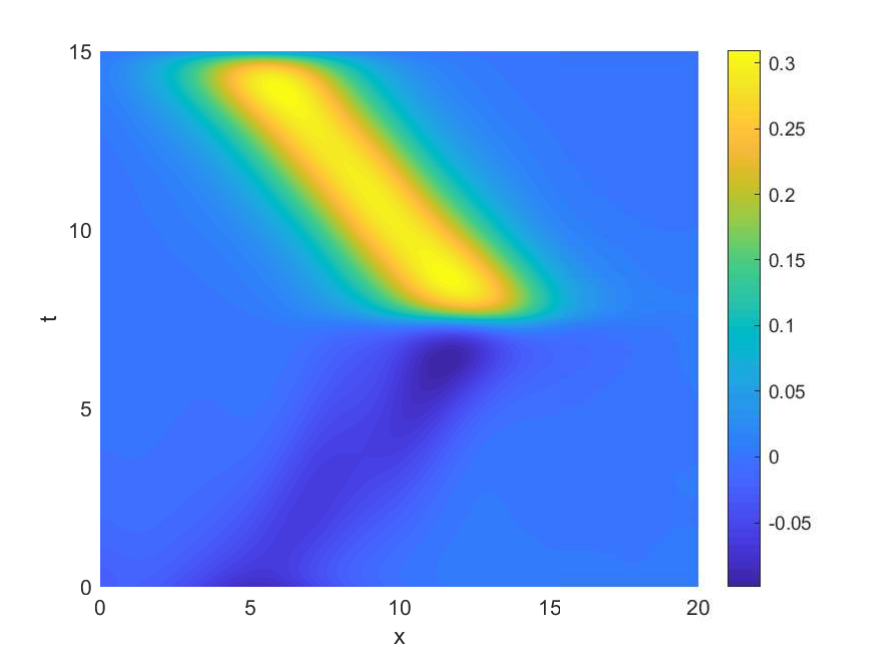}
		\caption{Optimal Control\newline\hfill}\label{10stochcontrolfunction}
	\end{minipage}
\end{figure}
\subsection{Comparison with the Control of the Deterministic System}\label{spdecase}
Simulations show that the optimal control for the deterministic system in the preceding example does not differ qualitatively from the optimal control for the stochastic system. This is because the fixed points $0$ and $1$ are stable. The situation changes, however, if one of the fixed points becomes unstable from one side, as the following example shows. Consider the state equation
\begin{equation}
	\begin{cases}
		\mathrm{d} u^{g}_t = \left [ \Delta u^{g}_t - (u^{g}_t)^3+( u^{g}_t)^2 + g(t) \right ] \mathrm{d} t + \sigma \mathrm{d}W^Q_t &\text{on } L^2(\Lambda)\\
		u^{g}_0(x) = u^0(x) &\text{in } \Lambda,
	\end{cases}
\end{equation}
where $\Lambda=[0,20]$, $T=30$ and $\sigma\in\mathbb{R}$. These choices lead to only one stable steady state, $u=1$ and one unstable steady state $u=0$. Now, as initial condition, we choose $u^g_0 = 0$, and consider the cost functional
\begin{equation}
	J(g) = \mathbb{E} \left [ \frac12 \int_{\Lambda} \left ( u^g_T(x) \right )^2 \mathrm{d}x \right ],
\end{equation}
i.e., we want the final state to be unchanged, in the unstable steady state $0$. In the deterministic case, the optimal control is clearly $g^{\ast}=0$, since we start in the steady state $x=0$ and without any forcing, we stay in this state and accomplish the minimal possible costs $J(g^{\ast})=0$. In the stochastic case, however, the noise term pushes the state out of the unstable steady state. Whenever the noise pushes the state above $0$, the dynamics of the state equation force the state towards the stable steady state $x=1$. As an illustration of this effect, Figure \ref{potential} displays the potential $F(x)$ of the nonlinearity $f$. Figure \ref{100stochnocontrol} shows one realization in the stochastic case without a control function.
\begin{figure}[H]
	\centering
	\begin{minipage}{0.45\textwidth}
		\centering
		\captionsetup{width=1\linewidth}
		\includegraphics[width=1.1\textwidth]{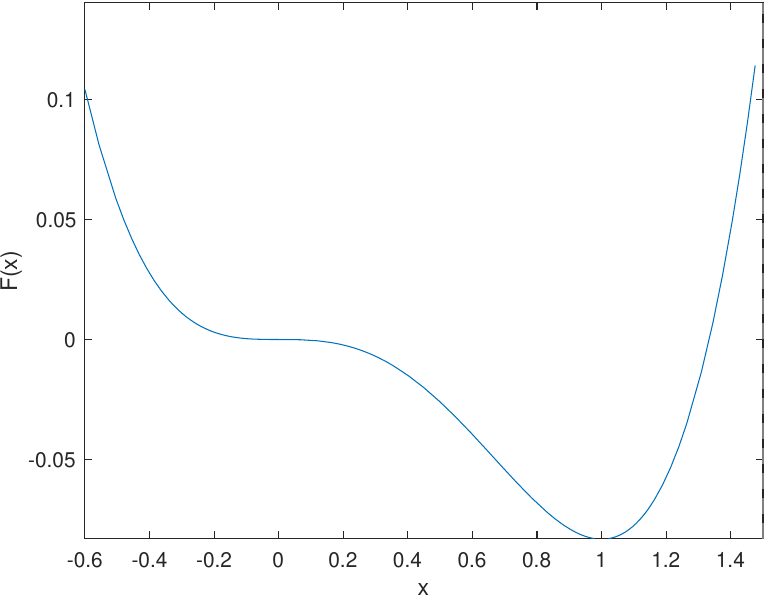}
		\caption{Potential\newline\hfill}\label{potential}
	\end{minipage}\hfill
	\begin{minipage}{0.45\textwidth}
		\centering
		\hspace*{-0.6cm}
		\captionsetup{width=1\linewidth}
		\includegraphics[width=1.1\textwidth]{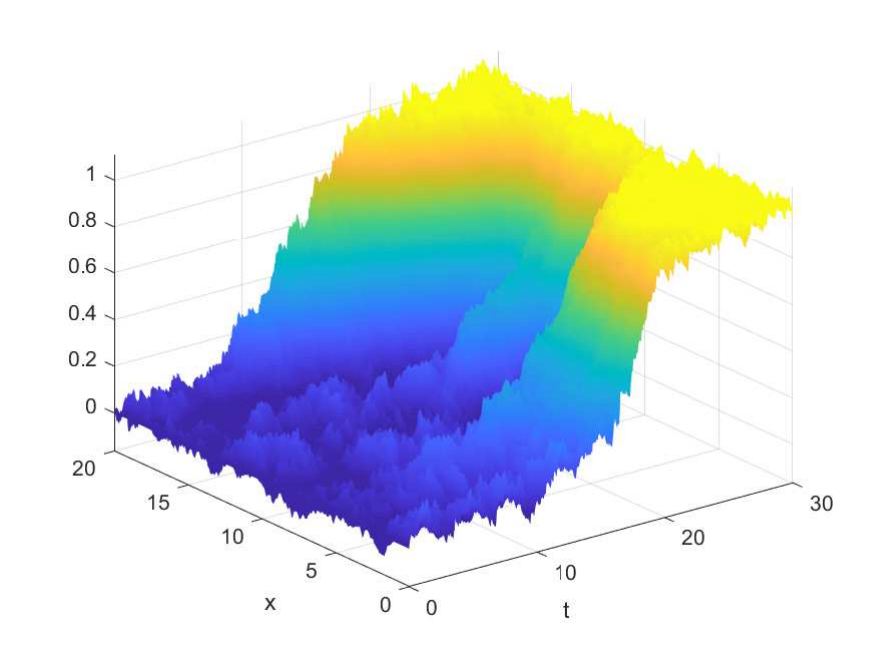}
		\caption{Solution without Control, $\sigma=1$}\label{100stochnocontrol}
	\end{minipage}
\end{figure}
When we introduce a control, the control tries to counteract this effect by keeping the state below $0$ for times $t<T$. This effect can be seen in the simulations, as well. As the stopping criterion we used $\eta=0.002$. Figures \ref{50stochoptcontrol} to \ref{50optimalstate} display the optimal controls in the stochastic case with $\sigma=0.5$ and one realization of the corresponding state.
\begin{figure}[H]
	\centering
	\begin{minipage}{0.45\textwidth}
		\centering
		\captionsetup{width=1\linewidth}
		\includegraphics[width=1.1\textwidth]{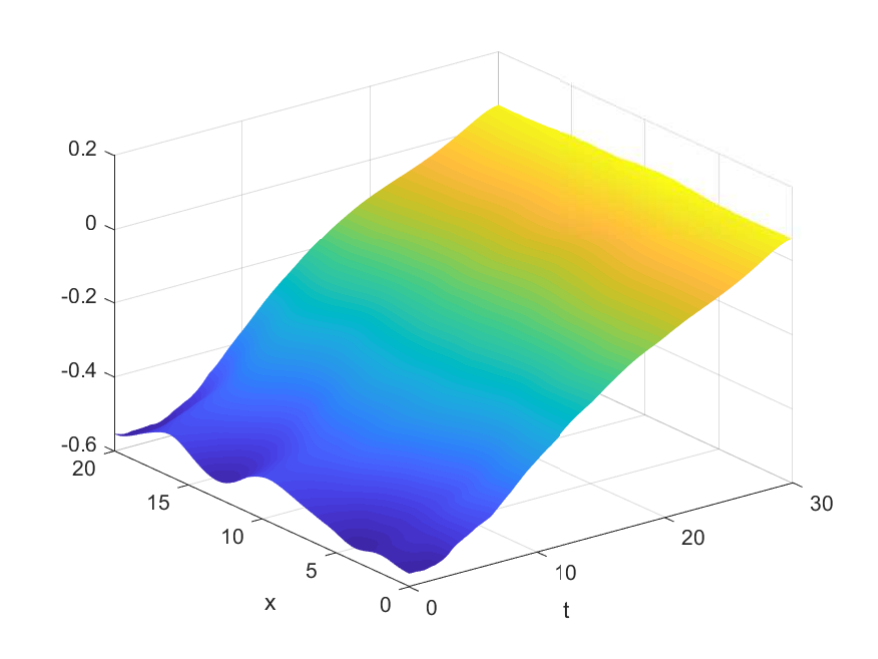}
		\caption{Optimal Control, $\sigma=0.5$\newline\hfill }\label{50stochoptcontrol}
	\end{minipage}\hfill
	\begin{minipage}{0.45\textwidth}
		\centering
		\hspace*{-1cm}
		\captionsetup{width=1\linewidth}
		\includegraphics[width=1.1\textwidth]{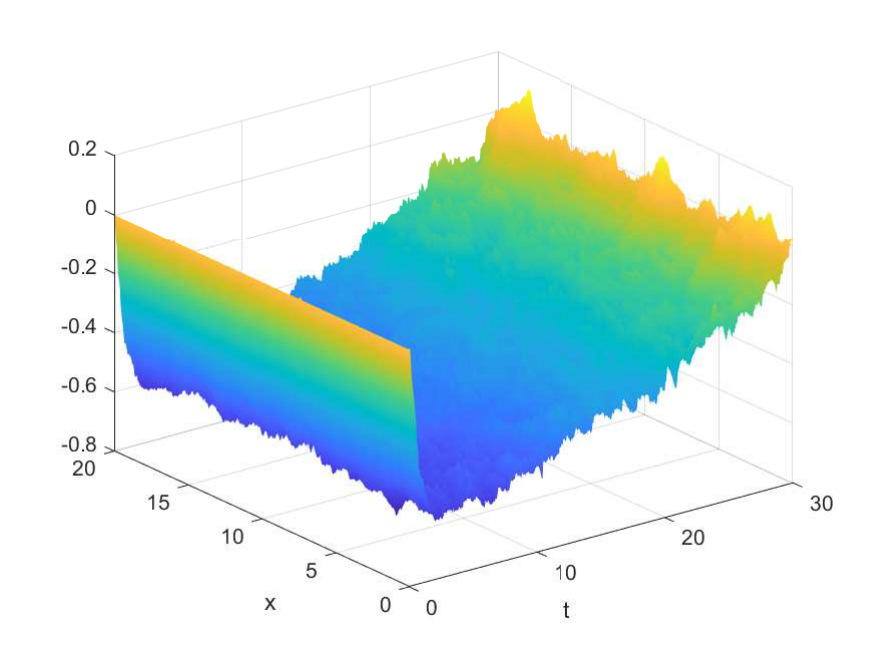}
		\caption{Solution with Optimal Control, $\sigma=0.5$}\label{50optimalstate}
	\end{minipage}
\end{figure}
\subsection{Mathematical Analysis in a Simplified Setting}\label{sdecase}

Since we are not able to prove the previous result in that setting rigorously, we consider a simpler similar example in which the optimal control in the deterministic case and the optimal control in the stochastic case differ.

Let us consider the stochastic ordinary differential equation
\begin{equation}\label{sde_state}
	\begin{cases}
		\mathrm{d}u^g_t = \left [ -V^{\prime}(u^g_t) + g(t) \right ] \mathrm{d}t + \sigma \mathrm{d}B_t, \quad t\in[0,T]\\
		u^g_0=0,
	\end{cases}
\end{equation}	
where $(B_t)_{t\geq 0}$ is a Brownian motion on $\mathbb{R}$, the potential $V:\mathbb{R}\to \mathbb{R}$ is given by
\begin{equation}
	V(x)=
	\begin{cases}
		\frac12 (\arctan(x) -x), &\text{for } x\geq 0\\
		0, &\text{for } x<0,
	\end{cases}
\end{equation}
and hence $-V^{\prime}$ is given by
\begin{equation}
	-V^{\prime}(x)=
	\begin{cases}
		\frac{x^2}{2(1+x^2)}, &\text{for } x\geq 0\\
		0, &\text{for } x<0.
	\end{cases}
\end{equation}
Notice that this potential qualitatively resembles the potential used in the previous example in the interval $[0,1]$. That is why we observe a similar effect in this example. We consider the cost functional
\begin{equation}\label{sde_costfunctional}
	J(g) := \mathbb{E} \left [ \frac12 \left ( u^g_T \right )^2 \right ].
\end{equation}

As in the previous example, the initial condition and the desired final state are both the unstable steady state $u=0$. Hence, in the deterministic case ($\sigma=0$), the optimal control is given by $g^{\ast} \equiv 0$, since the constant function $u\equiv 0$ solves the deterministic equation without control and the associated costs are zero.

Now, we are going to show that the optimal control in the stochastic case ($\sigma>0$), however, is not equal to zero. First, notice that the adjoint equation associated with our control problem is given by
\begin{equation}
	\begin{cases}
		-\partial_t p_t = -V^{\prime\prime} (u^g_t) p_t,\quad t\in[0,T]\\
		p_T=u^g_T,
	\end{cases}
\end{equation}
where $-V^{\prime\prime}$ is given by
\begin{equation}
	-V^{\prime\prime}(x)=
	\begin{cases}
		\frac{x}{(1+x^2)^2}, &\text{for } x\geq 0\\
		0, &\text{for } x<0.
	\end{cases}
\end{equation}
Hence, the solution of the adjoint equation is given explicitly by
\begin{equation}
	p_t = u^g_T \exp \left ( \int_t^T -V^{\prime\prime} (u^g_s) \mathrm{d}s \right ),
\end{equation}
and the gradient of the cost functional is given by
\begin{equation}
	\nabla J(g)(t) = \mathbb{E} [ p_t ] = \mathbb{E} \left [ u^g_T \exp \left ( \int_t^T -V^{\prime\prime}(u^g_s) \mathrm{d}s \right ) \right ].
\end{equation}
Now, we are going to show that the gradient for $g\equiv 0$ is not equal to zero and hence, $g\equiv 0$ is not an optimal control. To this end, consider
\begin{equation}
	\begin{split}
		\partial_t (\nabla J(g))(t) &= \mathbb{E} [ \partial_t p_t ] = \mathbb{E} \left [ V^{\prime\prime}(u^g_t) u^g_T \exp \left ( \int_t^T -V^{\prime\prime}(u^g_s) \mathrm{d}s \right ) \right ].
	\end{split}
\end{equation}
This yields
\begin{equation}
	\begin{split}
		&\liminf_{t\to T} \left \{ -\partial_t (\nabla J(g))(t) \right \}\\
		=&\liminf_{t\to T} \mathbb{E} \left [ - V^{\prime\prime}(u^g_t) u^g_T \exp \left ( \int_t^T -V^{\prime\prime}(u^g_s) \mathrm{d}s \right ) \right ]\\
		\geq &\mathbb{E} \left [ \liminf_{t\to T} \left \{- V^{\prime\prime}(u^g_t) u^g_T \exp \left ( \int_t^T -V^{\prime\prime}(u^g_s) \mathrm{d}s \right ) \right \} \right ]\\
		=& \mathbb{E} \left [ - V^{\prime\prime}(u^g_T) u^g_T \right ]\\
		=& \mathbb{E} \left [ \frac{(u^g_T)^2}{\left (1+\left ( u^g_T \right )^2 \right )^2} 1_{\{u^g_T>0\}} \right ] >0,
	\end{split}
\end{equation}
where the last part is strictly positive since $u_T$ has a strictly positive density with respect to the Lebesgue measure. Therefore, the gradient is not equal to zero and thus, $g\equiv 0$ is not an optimal control.

\begin{remark}
	Notice that we did not use that $g\equiv 0$ in this proof. This shows, that the optimal control in the stochastic case is unbounded.
\end{remark}

Figures \ref{sde_optcontrol} and \ref{sde_optstate} illustrate our results in case of the stochastic ordinary differential equation \eqref{sde_state} as the constraint and the cost functional \eqref{sde_costfunctional}.

\begin{figure}[H]
	\begin{minipage}{0.45\textwidth}
		\centering
		\captionsetup{width=1.05\linewidth}
		\includegraphics[width=1.05\linewidth]{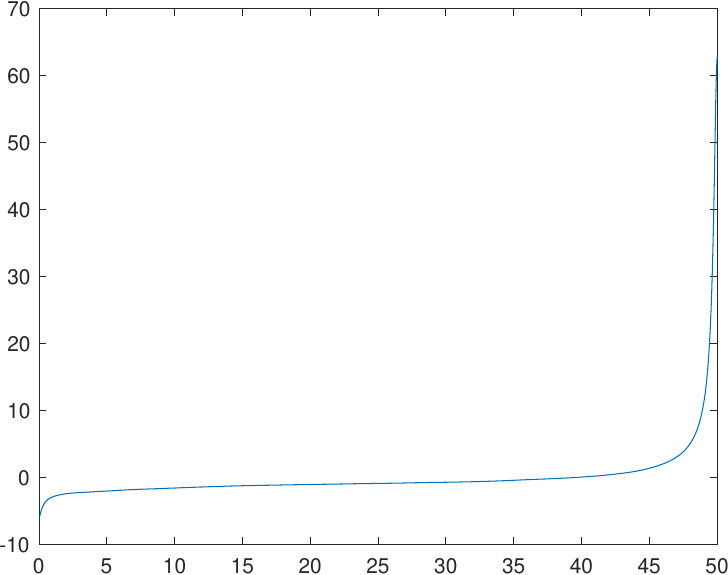}
		\caption{Optimal Control, $\sigma=1$\newline \hfill}\label{sde_optcontrol}
	\end{minipage}\hfill
	\begin{minipage}{0.45\textwidth}
		\centering
		\hspace*{-0.3cm}
		\captionsetup{width=1.05\linewidth}
		\includegraphics[width=1.05\linewidth]{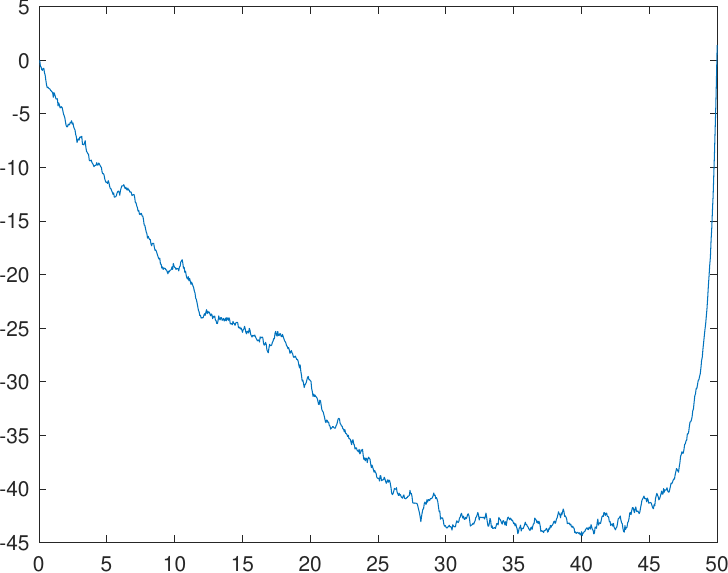}
		\caption{Solution with Optimal Control, $\sigma=1$}\label{sde_optstate}
	\end{minipage}
\end{figure}

\section*{Acknowledgement} This work has been funded by Deutsche Forschungsgemeinschaft (DFG) through grant CRC 910 ``Control of self-organizing nonlinear systems: Theoretical methods and concepts of application,'' Project (A10) ``Control of stochastic mean-field equations with applications to brain networks.'' The authors would like to thank the referee for their valuable feedback.



\end{document}